\documentclass[10pt]{amsart}
\usepackage{amsmath}
\usepackage{amssymb}
\usepackage{amsfonts}
\usepackage{pxfonts}
\usepackage[OT2,T1]{fontenc}
\usepackage[
textwidth=16.0cm, 
textheight=22.0cm,
hmarginratio=1:1,
vmarginratio=1:1]{geometry}
\usepackage{pifont}
\usepackage[english]{babel}
\usepackage{graphicx}
\usepackage{mathbbol}

\newtheorem{thm}{Theorem}[subsection]
\newtheorem{cor}[thm]{Corollary}
\newtheorem{lem}[thm]{Lemma}
\newtheorem{prop}[thm]{Proposition}
\theoremstyle{definition}
\newtheorem{defn}[thm]{Definition}

\theoremstyle{remark}
\newtheorem{rem}[thm]{Remark}

\numberwithin{equation}{subsection}
\numberwithin{figure}{subsection}

\newcommand{\diff}{\mathrm{d}}
\newcommand{\C}{{\mathbb C}}
\newcommand{\R}{{\mathbb R}}
\newcommand{\D}{{\mathbb D}}
\newcommand{\expect}{{\mathbb E}}
\newcommand{\Te}{{\mathbb T}}
\newcommand{\Z}{{\mathbb Z}}

\newcommand{\imag}{\mathrm{i}}
\newcommand{\e}{\mathrm{e}}

\newcommand\calD{\mathcal{D}}
\newcommand\calH{\mathcal{H}}
\newcommand\calX{\mathcal{X}}
\newcommand\calB{\mathcal{B}}
\newcommand\Gspace{\mathfrak{G}}
\newcommand\Hspace{\mathfrak{H}}
\newcommand{\Xspace}{\mathfrak{A}}
\newcommand{\Yspace}{\mathfrak{B}}
\newcommand{\Nspace}{\mathfrak{N}}
\newcommand{\Mspace}{\mathfrak{M}}
\newcommand{\Pop}{{\mathbf P}}

\newcommand{\boldG}{{\boldsymbol\Gamma}}
\newcommand{\Qop}{{\mathbf Q}}

\newcommand{\Mop}{{\mathbf M}}
\newcommand{\Lop}{{\mathbf L}}
\newcommand{\Tope}{{\mathbf T}}
\newcommand{\Uop}{{\mathbf U}}
\newcommand{\Sope}{{\mathbf S}}
\newcommand{\Vop}{{\mathbf V}}

\newcommand{\Ordo}{\mathrm{O}}

\newcommand{\Sop}{\mathbf{S}}

\newcommand{\Zsymb}{\mathcal{P}}
\newcommand{\Wsymb}{\mathcal{Q}}

\newcommand{\Dop}{\mathbf{D}}
\newcommand{\Iop}{\mathbf{I}}
\newcommand{\Aop}{\mathbf{A}}
\newcommand\Rop{\mathbf{R}}

\newcommand{\Bop}{\mathbf{B}}

\newcommand{\kernel}{\mathrm k}
\newcommand{\balpha}{\boldsymbol\alpha}
\newcommand{\szeg}{\mathrm s}

\DeclareMathOperator{\re}{Re}

\begin{document}

%
\title{Gaussian analytic functions and operator symbols of Dirichlet type}

\author{Haakan Hedenmalm}
\address{
Hedenmalm: Department of Mathematics\\
KTH Royal Institute of Technology\\
S--10044 Stockholm\\
Sweden}

\email{haakanh@math.kth.se}

\author{Serguei Shimorin}
\address{
Shimorin: Department of Mathematics\\
KTH Royal Institute of Technology\\
S--10044 Stockholm\\
Sweden}

\email{shimorin@math.kth.se}

\subjclass[2000]{Primary 30C40, 30B20, 30C55, 60G55}
\keywords{}
 
\thanks{This research was supported by Vetenskapsr\aa{}det (VR)}
 
\begin{abstract} 

Let $\calH$ be a separable infinite-dimensional $\C$-linear Hilbert space,
with sesquilinear inner product $\langle\cdot,\cdot\rangle_\calH$.
Given any two orthonormal systems $x_1,x_2,x_3,\ldots$ and 
$y_1,y_2,y_3,\ldots$ in $\calH$, we show that
\[
\sum_{l=2}^{+\infty}s^{l}\bigg|
\sum_{j,k:j+k=l}(jk)^{-\frac12}\langle x_j,y_k\rangle_{\calH}\bigg|^2
\le 2s\log\frac{\e^{1/2}}{1-s},\qquad0\le s<1 .
\leqno{\text(1)}
\]
In terms of the weighted sums
\[
S(l):=\sum_{j,k:j+k=l}\bigg(\frac{l}{jk}\bigg)^{\frac12}\,
\langle x_j,y_k\rangle_{\calH},
\]
this means that
\[
\sum_{l=2}^{+\infty}\frac{s^{l}}{l}|S(l)|^2\le 
2\log\frac{\e^{1/2}}{1-s},\qquad0\le s<1. 
\]
Expressed more vaguely, $|S(l)|^2\lessapprox2$ holds in the sense of averages.
Concerning the optimality of the bound (1), a construction due to Zachary 
Chase shows that the statement does not hold if the number $2$ is replaced 
by the smaller number $1.72$. In the construction, the 
system $y_1,y_2,y_3,\ldots$ is a permutation of the system $x_1,x_2,x_3,\ldots$. 
We interpret our bound in terms of the correlation $\expect \Phi(z)\Psi(z)$
of two copies of a Gaussian analytic function with possibly intricate Gaussian
correlation structure between them. The Gaussian analytic function we study
arises in connection with the classical Dirichlet space, which is naturally 
M\"obius invariant. The study of the correlations 
$\expect\Phi(z)\Psi(z)$ 
leads us to introduce a new space, the \emph{mock-Bloch space},
which is slightly bigger than the standard Bloch space. Our bound has
an interpretation in terms of McMullen's asymptotic variance, originally 
considered for functions in the Bloch space. Finally, we show that the 
correlations $\expect\Phi(z)\Psi(w)$ may be expressed as Dirichlet symbols of
contractions on $L^2(\D)$, and show that the Dirichlet symbols of Grunsky
operators associated with univalent functions find a natural characterization
in terms of a nonlinear wave equation.   
\end{abstract}

\maketitle

\section{Introduction} 

\subsection{Basic notation in the plane}
\label{subsec-1.1}
We write $\Z$ for the integers, $\Z_+$ for the positive integers, $\R$ 
for the real line, and $\C$ for the complex plane. Moreover, we write 
$\C_\infty:=\C\cup\{\infty\}$ for the extended complex plane 
(the Riemann sphere). For a complex variable $z=x+\imag y\in\C$, let 
\[
\diff s(z):=\frac{|\diff z|}{2\pi},\qquad
\diff A(z):=\frac{\diff x\diff y}{\pi},
\]
denote the normalized arc length and area measures, as indicated. 
Moreover, we shall write 
\[
\varDelta_z:=\frac{1}{4}\bigg(\frac{\partial^2}{\partial x^2}+
\frac{\partial^2}{\partial y^2}\bigg)
\]
for the normalized Laplacian, and
\[
\partial_z:=\frac{1}{2}\bigg(\frac{\partial}{\partial x}-\imag
\frac{\partial}{\partial y}\bigg),\qquad
\bar\partial_z:=\frac{1}{2}\bigg(\frac{\partial}{\partial x}+\imag
\frac{\partial}{\partial y}\bigg),
\]
for the standard complex derivatives; then $\varDelta$ factors as
$\varDelta_z=\partial_z\bar\partial_z$. Often we will drop the subscript 
for these differential operators when it is obvious from the context with
respect to which variable they apply.
We let $\D$ denote the open unit disk, $\Te:=\partial\D$ the unit circle, 
and $\D_e$ the exterior disk:
\[
\D:=\{z\in\C:\,\,|z|<1\},\qquad \D_e:=\{z\in\C_\infty:\,\,|z|>1\}.
\]
We will find it useful to introduce the sesquilinear forms 
$\langle\cdot,\cdot \rangle_\C$ and $\langle\cdot,\cdot \rangle_\D$,
as given by
\[
\langle f,g \rangle_\C:=\int_\C f(z)\bar g(z)\diff A(z),\qquad
\langle f,g\rangle_\D:=\int_\D f(z)\bar g(z)\diff A(z),
\] 
where we need $f\bar g\in L^1(\C)$ in the first instance and 
$f\bar g\in L^1(\D)$ in the second. These are standard Lebesgue spaces 
with respect to normalized area measure $\diff A$. 
Here, generally, for a given complex-valued function $f$, we denote by 
$\bar f$ the function whose values are the complex conjugates of $f$.
To simplify the notation further, we write
\[
\langle f\rangle_\C=\langle f,1\rangle_\C,\quad
\langle f\rangle_\D=\langle f,1\rangle_\D.
\]
As for operators $\Tope$ on a Hilbert function space, we let $\Tope^*$ denote
the adjoint, while $\bar\Tope$ means the operator defined by
\[
\bar\Tope f=\overline{\Tope\bar f}.
\]



\subsection{Complex Gaussian Hilbert space}
\label{subsec-CGHS}

A \emph{Gaussian Hilbert space} is a closed linear subspace $\Gspace$ of 
$L^2(\Omega)=L^2(\Omega,\diff P)$, where $(\Omega,\diff P)$ is a probability
space with a given $\sigma$-algebra, with the property that each element
$\gamma\in\Gspace$ has a Gaussian distribution with mean $0$. Since we will
be working with the complex field $\C$, this means that the real and imaginary
parts of $\gamma$ are jointly Gaussian, and that the mean is $0$ of each one.
Here, the \emph{expectation} (or \emph{mean}) operation $\expect$ is just
given by $\expect \gamma:=\langle\gamma\rangle_\Omega=\int_\Omega\gamma\diff P$.
We say that $\gamma$ is \emph{symmetric} if $\expect(\gamma^2)=0$. Moreover,
$\gamma$ is a \emph{standard complex Gaussian} variable if it has mean $0$, 
is symmetric and has $\expect(|\gamma|^2)=1$.
In other words, the values of $\gamma$ are distributed according to the 
density $\e^{-|z|^2}\diff A(z)$ in the plane. 
We will assume for convenience that $\Gspace$ is \emph{conjugation-invariant}, 
that is, $\gamma\in\Gspace\Longleftrightarrow\bar\gamma\in\Gspace$. We refer
to \cite{Janson} for an exposition on Gaussian Hilbert spaces.
We will write $\langle\gamma,\gamma'\rangle_\Omega
=\langle\gamma\bar\gamma'\rangle_\Omega=\expect \gamma\bar\gamma'$
for the inner product of $\Gspace$. We shall need the following observation. 
If $\Gspace$ is separable and infinite-dimensional, then there exists a 
sequence $\gamma_1,\gamma_2,\gamma_3,\ldots$ in $\Gspace$ consisting of 
i i d standard complex Gaussians, such that the sequence 
$\gamma_1,\bar\gamma_1,\gamma_2,\bar\gamma_2,\ldots$ forms an orthonormal basis
in $\Gspace$. In particular, $\Gspace$ splits as an orthogonal sum
$\Gspace=\Hspace\oplus\Hspace_*$,  where $\Hspace$ is the closed 
subspace spanned by $\gamma_1,\gamma_2,\gamma_3,\ldots$,  while $\Hspace_*$
is spanned by $\bar\gamma_1,\bar\gamma_2,\bar\gamma_3,\ldots$.

\subsection{Gaussian analytic functions associated with the 
Dirichlet space}
We now outline a more direct approach to the analytic part of GFF outlined
in the preceding subsection. 
Let $A^2(\D)$ denote the subspace of $L^2(\D)$ consisting of the holomorphic
functions, which is a closed subspace and hence a Hilbert space in its own
right, known as the \emph{Bergman space}.
The \emph{Dirichlet space} is the space $\calD(\D)$ of analytic functions $f$ 
with $f'\in A^2(\D)$, equipped with the Dirichlet inner product
\[
\langle f,g\rangle_\nabla:=\langle f',g'\rangle_\D,
\]
The importance of the Dirichlet space comes from its conformal invariance 
property. For instance, if $\phi$ is a M\"obius automorphism of the unit disk
$\D$, we have that
\[
\langle f\circ\phi,g\circ\phi\rangle_\nabla=\langle f,g\rangle_\nabla.
\]
The Dirichlet inner product gives rise to a seminorm
\[
\|f\|_\nabla^2:=\|f'\|^2_{A^2(\D)}=\langle f',f'\rangle_\D,
\]
which vanishes on the constant functions. So, to make it a norm, we
could add the requirement that the functions should vanish at a given point
$\lambda\in\D$:
\[
\calD_\lambda(\D):=\{f\in\calD(\D):\,f(\lambda)=0\}. 
\]
We will focus our attention to $\lambda=0$, and study the space $\calD_0(\D)$.
By the M\"obius invariance of the seminorm, this choice is not restrictive 
as we may easily move any other point $\lambda$ to the origin using a 
M\"obius automorphism.

In recent years, \emph{Gaussian analytic functions} has received increasing
attention. For instance, see \cite{Sodin} and the book \cite{HKPV}. 
In the space $\calD_0(\D)$, we have a canonical orthogonal basis
\[
e_j(z):=j^{-\frac12}z^j,\qquad j=1,2,3,\ldots,
\]   
and we form a $\calD_0$-Gaussian analytic function ($\calD_0$-GAF)
\begin{equation}  
\Phi(z):=\sum_{j=1}^{+\infty}\alpha_j\,e_j(z)=
\sum_{j=1}^{+\infty}\frac{\alpha_j}{\sqrt{j}}\,z^j,
\label{eq-Gauss1}
\end{equation}
where the $\alpha_j$ are i i d (independent identically distributed) standard
complex Gaussian variables, taken from a Gaussian Hilbert space $\Gspace$. 
Then for two points in the disk $z,w\in\D$, we have the complex correlation 
structure
\begin{equation}
\expect(\Phi(z)\Phi(w))=0, \quad \expect(\Phi(z)\bar\Phi(w))=
\log\frac{1}{1-z\bar w}. 
\label{eq-Gauss2}
\end{equation}
Since Gaussian random variables are determined by their correlation structures, 
we may, depending on the point of view, take \eqref{eq-Gauss2} 
as the defining property instead of the more explicit \eqref{eq-Gauss1}.
On the right-hand side of \eqref{eq-Gauss2}, we recognize the reproducing
kernel for the Dirichlet space,
\begin{equation}
\kernel_{\calD_0}(z,w)=\log\frac{1}{1-z\bar w},
\end{equation}
with the point evaluation property
\[
f(w)=\langle f,\kernel_{\calD_0}(\cdot,w)\rangle_\nabla,\qquad f\in \calD_0(\D).
\]
It is appropriate to think of the correlation structure \eqref{eq-Gauss2} in
terms of the matrix-valued correlation structure
\begin{equation}
\mathbb{k}_{2\times2}[\Phi](z,w)=\expect
\begin{pmatrix}
\Phi(z)\\
\bar \Phi(z)
\end{pmatrix}
\begin{pmatrix}
\bar\Phi(w) & \Phi(w)
\end{pmatrix}=\begin{pmatrix}
\expect \Phi(z)\bar\Phi(w) & 
\expect\Phi(z)\Phi(w)\\
\expect \bar\Phi(z)\bar\Phi(w) & 
\expect\bar\Phi(z)\Phi(w)
\end{pmatrix}
=
\begin{pmatrix}
\log\frac{1}{1-z\bar w} & 0\\
0 & 
\log\frac{1}{1-\bar z w}
\end{pmatrix},
\label{eq-matrixcorr1}
\end{equation}
and the associated $4\times4$ matrix 
\begin{equation}
\begin{pmatrix}
\mathbb{k}_{2\times2}[\Phi](z,z)&\mathbb{k}_{2\times2}[\Phi](z,w)
\\
\mathbb{k}_{2\times2}[\Phi](z,w)^\ast &\mathbb{k}_{2\times2}[\Phi](w,w)
\end{pmatrix}
\label{eq-matrixcorr1.1}
\end{equation}
is positive semidefinite (the asterisque $\ast$ stands for the operation of
taking the adjoint of the matrix). 
The real part of $\Phi(z)$ may be understood, up to an additive constant, 
as the restriction of the Gaussian free field (GFF) on $\C$ conditioned to be 
harmonic in $\D$. For some background on GFF, we refer to the survey paper 
\cite{Sheff} as well as to \cite{HedNiem}. 
Alternatively, the process $\Phi(z)$ may be identified
as the limit of the logarithm of the characteristic polynomial for random
unitary matrices as the size of the matrices tends to infinity (see below). 

In analogy with \cite{PerVir}, it might be of interest to study the random 
zeros of the function $\Phi(z)$, but since one of them is deterministic
(the origin), we should not expect full M\"obius automorphism invariance. 
By the Edelman-Kostlan formula (see \cite{Sodin}) the density of
zeros is given by
\begin{equation}
\varDelta\log \kernel_{\calD_0}(z,z)\,\diff A(z)
=\varDelta\log\log\frac{1}{1-|z|^2}
\diff A(z),
\end{equation} 
which has a unit point mass at the origin due to the deterministic zero there. 
Here, one might also be interested in the process for the critical points. 
We will not pursue any of these directions here. A rather interesting object 
appears to be the random curve (or tree) structure we obtain by following the 
gradient flow for the random harmonic function $\re\Phi(z)$ which stops at 
critical points. At each critical point we would instead choose among the 
possible directions, for instance by maximizing the second directional 
derivative (perhaps after precomposing with a M\"obius mapping to put 
the critical point at the origin). Although quite promising, We will not 
pursue this matter further here. A related setting of gradient flow for the 
plane defined in terms of the Bargmann-Fock space was studied by Nazarov,
Sodin, and Volberg \cite{NSV}.

\subsection{$\calD_0$-Gaussian analytic functions and 
random unitary matrices}
Let $M_n$ be a random $n\times n$ unitary matrix with distribution given 
by Haar measure. Let
\[
\chi_{M_n}(\lambda)=\det(\lambda{I}_n-M_n)
\]
be the associated random characteristic polynomial, where $I_n$ is the 
$n\times n$ identity matrix. Diaconis and Evans \cite{DiacEv} found an 
interesting relationship connecting the characteristic polynomial of $M_n$ 
with the process given by \eqref{eq-Gauss1}. They showed that
\[
\operatorname{tr}\log(I_n-zM_n^*)=\log\det(I_n-zM_n^*)
=\log\frac{\chi_{M_n}(z)}{\chi_{M_n}(0)}
\] 
converges, as $n\to+\infty$, in distribution, to the $\calD_0$-Gaussian 
analytic function $\Phi(z)$ given by \eqref{eq-Gauss1}. The details are 
supplied in Example 5.6 of \cite{DiacEv}. 
For the convenience of the reader, we mention
that the master relationship between their random function $F_n(z)$ and 
$\chi_{M_n}(z)$ has a typo, and should be replaced by
\[
F_n(z)=\frac{n}{2\pi}-\frac{z}{\pi}\frac{\chi_{M_n}'(z)}{\chi_{M_n}(z)}.
\] 

\begin{rem}
The matters considered here, the possible correlation structure of two 
jointly Gaussian $\calD_0$-GAFs, have their (finite-dimensional) 
counterpart for random matrices. Let $M_n$ and $M_n'$ be two copies of
the random $n\times n$ unitary matrix enemble, with possibly complicated
correlation structure between $M_n$ and $M_n'$, but at least all their
entries are jointly (complex) Gaussian variables. What could we say about the 
structure of the $\C^2$-valued process of normalized 
random characteristic polynomials
\[
\bigg(\frac{\chi_{M_n}(z)}{\chi_{M_n}(0)},
\frac{\chi_{M_n'}(z)}{\chi_{M_n'}(0)}\bigg)?
\]   
 
\end{rem}

\subsection{Two interacting copies of the $\calD_0$-Gaussian 
analytic function process}

The topic here involves two copies of the process \eqref{eq-Gauss1},
\begin{equation}
\Phi(z):=\sum_{j=1}^{+\infty}\frac{\alpha_j}{\sqrt{j}}\,z^j,\quad
\Psi(z):=\sum_{j=1}^{+\infty}\frac{\beta_j}{\sqrt{j}}\,z^j,
\label{eq-Gauss2.1}
\end{equation}
where $\Phi(z)$ is as before and the $\beta_j$ are i i d from $N_{\C}(0,1)$,
taken from the same Gaussian Hilbert space $\Gspace\subset L^2(\Omega)$.
We will refer to $(\Phi(z),\Psi(z))$ as a \emph{pair of jointly Gaussian 
$\calD_0$-GAFs}. Consisting of jointly Gaussian variables with zero mean, the  
vector-valued process $(\Phi(z),\Psi(z))$ is governed by the correlation matrix
\begin{multline}
{\mathbb{k}}_{4\times4}[\Phi,\Psi](z,w):=
\expect\begin{pmatrix}
\Phi(z)\\
\bar\Phi(z)\\
\Psi(z)\\
\bar\Psi(z)
\end{pmatrix}
\begin{pmatrix}
\bar\Phi(w)&\Phi(w)&\bar\Psi(w)&\Psi(w)
\end{pmatrix}
\\
=\begin{pmatrix}
\expect\Phi(z)\bar\Phi(w) & \expect\Phi(z)\Phi(w) & 
\expect\Phi(z)\bar\Psi(w) &\expect\Phi(z)\Psi(w)
\\
\expect\bar\Phi(z)\bar\Phi(w) & \expect\bar\Phi(z)\Phi(w) & 
\expect\bar\Phi(z)\bar\Psi(w) & \expect\bar\Phi(z)\Psi(w)
\\
\expect\Psi(z)\bar\Phi(w) & \expect\Psi(z)\Phi(w) & 
\expect\Psi(z)\bar\Psi(w) & \expect\Psi(z)\Psi(w)
\\
\expect\bar\Psi(z)\bar\Phi(w) & \expect\bar\Psi(z)\Phi(w) & 
\expect\bar\Psi(z)\bar\Psi(w) & \expect\bar\Psi(z)\Psi(w)
\end{pmatrix}
\\
=\begin{pmatrix}
\log\frac{1}{1-z\bar w} & 0 & 
\expect\Phi(z)\bar\Psi(w) &\expect\Phi(z)\Psi(w)
\\
0 & \log\frac{1}{1-\bar z w} & 
\expect\bar\Phi(z)\bar\Psi(w) & \expect\bar\Phi(z)\Psi(w)
\\
\expect\Psi(z)\bar\Phi(w) & \expect\Psi(z)\Phi(w) & 
\log\frac1{1-z\bar w} & 0
\\
\expect\bar\Psi(z)\bar\Phi(w) & \expect\bar\Psi(z)\Phi(w) & 
0 & \log\frac{1}{1-\bar z w}
\end{pmatrix},
\label{eq-4X4}
\end{multline}
and the associated $8\times8$ matrix
\begin{equation}
\begin{pmatrix}
\mathbb{k}_{4\times4}[\Phi](z,z)&\mathbb{k}_{4\times4}[\Phi](z,w)
\\
\mathbb{k}_{4\times4}[\Phi](z,w)^\ast &\mathbb{k}_{4\times4}[\Phi](w,w)
\end{pmatrix}
\label{eq-4X4.1}
\end{equation}
is positive semidefinite.
Note that although there are eight unknown entries in \eqref{eq-4X4}, in 
fact only two are needed, as clearly,
\[
\expect(\bar\Phi(z)\bar\Psi(w))=
\overline{\expect(\Phi(z)\Psi(w))},\quad
\expect(\bar\Phi(z)\Psi(w))=\overline{\expect(\Phi(z)\bar\Psi(w))},
\]
and the remaining four only involve exchanging the variables $z$ and $w$.

So we need only be concerned with the quantities
\begin{equation}
\expect(\Phi(z)\bar\Psi(w))\quad\text{and}\quad \expect(\Phi(z)\Psi(w)).
\label{eq-correl1.01}
\end{equation}
In a sense they complement each other, as we see below. 

\begin{prop}
We have that
\[
|\expect\Phi(z)\bar\Psi(w)|+|\expect\Phi(z)\Psi(w)|\le
\bigg(\log\frac{1}{1-|z|^2}\bigg)^{\frac12}
\bigg(\log\frac{1}{1-|w|^2}\bigg)^{\frac12},
\qquad z,w\in\D.
\]
\label{prop-triang1}
\end{prop}

Since for a given point with $|z|=|w|$ each of the two terms on the 
left-hand side may reach up to the right-hand side bound, the estimate tells
us they cannot do so simultaneously. 
The proof of this estimate is presented in Subsection \ref{subsec-thm-fund1}.

\subsection{The fundamental integral estimate}

The following is our basic estimate of the correlations. 

\begin{thm}
For $a,b\in\C$, we have the estimate
\[
\int_\D\big|a w\expect\Phi(z)\Psi'(w)+b\bar w\expect\Phi(z)\bar\Psi'(w)\big|^2
\frac{\diff A(w)}{|w|^2}\le (|a|^2+|b|^2)\log\frac{1}{1-|z|^2},\qquad z\in\D.
\]
\label{thm-fund1}
\end{thm}

This may be interpreted as an estimate of the radial derivative (with respect 
to $w$) of the harmonic function 
\[
a\expect\Phi(z)\Psi(w)+b\expect\Phi(z)\bar\Psi(w).
\]
Indeed, if $F$ is holomorphic in $\D$, then its radial derivative is
\[
\partial_r F(r\e^{\imag\theta})=\e^{\imag\theta}F'(r\e^{\imag\theta}),
\]
so that the estimate of Theorem \ref{thm-fund1} asserts that ($\partial_{r(w)}$
is the radial derivative in the $w$ variable)
\begin{equation}
\int_\D\big|\partial_{r(w)}\big(a \expect\Phi(z)\Psi(w)+b
\expect\Phi(z)\bar\Psi(w)\big)\big|^2
\diff A(w)
\le (|a|^2+|b|^2)\log\frac{1}{1-|z|^2},\qquad z\in\D.
\label{eq-fundest2}
\end{equation}
Interesting estimates are obtained for instance when $(a,b)=(1,0)$ and
$(a,b)=(0,1)$. We shall mainly focus on the first of these, when 
$(a,b)=(1,0)$. We defer the proof of this result to Section 
\ref{sec-hilbspaces}.

\subsection{Growth of correlations in the mean along diagonals}
We are interested in the behavior of the correlations 
\[
\expect\Phi(z)\Psi(w),\quad\expect\Phi(z)\bar\Psi(w)
\]
as $z,w\in\D$ approach the unit circle $\Te$. The first one we will refer to
as the \emph{analytic correlation}, and the second the 
\emph{sesquianalytic correlation}.
We may study the growth behavior by looking along complex lines through 
the origin $w=\lambda z$ for some parameter $\lambda\in\C$ in which case 
our correlations are 
\begin{equation}
\expect\Phi(z)\Psi(\lambda z),\quad\expect\Phi(z)\bar\Psi(\lambda z).
\label{eq-twocorr1}
\end{equation}
The alternative study of conjugate-linear lines $w=\mu\bar z$ with $\mu\in\C$
is completely analogous and essentially only corresponds to reversing 
the order of these correlations (in the sense that 
$w\mapsto\bar\Psi(\mu\bar w)$ is a GAF). For this reason we will not consider
such conjugate-linear lines further. When $|\lambda|<1$ the process $\Phi(z)$
dominates in the correlations since $\Psi(\lambda z)$ is analytic in the disk
$\D(0,|\lambda|^{-1})$, while if $|\lambda|>1$ instead the process 
$\Psi(\lambda z)$ dominates. The most interesting instance seems to be the 
balanced case when $|\lambda|=1$, in which case the line $w=\lambda z$ might 
be called a \emph{generalized diagonal}.  For $|\lambda|=1$, the process 
$\Psi(\lambda z)$ is just another copy of the $\calD_0$-GAF, so as long as
$\lambda$ is fixed we might as well consider $\lambda=1$. So the study of 
\eqref{eq-twocorr1} for fixed $\lambda$ with $|\lambda|=1$ reduces to the
diagonal case
\begin{equation}
\expect\Phi(z)\Psi(z),\quad\expect\Phi(z)\bar\Psi(z).
\label{eq-twocorr2}
\end{equation}
We note that by Proposition \ref{prop-triang1},
\begin{equation}
|\expect\Phi(z)\Psi(z)|+|\expect\Phi(z)\bar\Psi(z)|\le\log\frac{1}{1-|z|^2}.
\label{eq-twocorr3}
\end{equation}
Some examples should elucidate which term, if any, may be dominant on the 
left-hand side. 

\begin{rem} We supply some examples which help us understand the size
of the two contributions on the left-hand side of \eqref{eq-twocorr3}. 

\noindent (a) If $\Psi=\Phi$, then
\[
\expect\Phi(z)\Psi(z)=\expect(\Phi(z)^2)=0,\quad
\expect\Phi(z)\bar\Psi(z)=\expect|\Phi(z)|^2=\log\frac{1}{1-|z|^2}.
\]
In this case we have \emph{equality} in \eqref{eq-twocorr3}, 
and on the left-hand side the first term vanishes, while the second is
dominant.

\noindent{(b)}
If $\Psi(z)$ and $\Phi(z)$ are stochastically independent, we have 
\[
\expect\Phi(z)\bar\Psi(z)=
\expect\Phi(z)\Psi(z)=0,
\]
so that both contributions to the left-hand side \eqref{eq-twocorr3}
collapse.

\noindent{(c)}
Consider $\Psi(z)=\bar\Phi(\bar z)$, when
\[
\expect\Phi(z)\Psi(z)=\expect\Phi(z)\bar\Phi(\bar z)=\log\frac{1}{1-z^2},
\quad \expect\Phi(z)\bar\Psi(z)=\expect\Phi(z)\Phi(\bar z)=0.
\] 
So at least pointwise, $\expect\Phi(z)\Psi(z)$ may be the dominant 
contribution in \eqref{eq-twocorr3}.
\label{rem-1.1}
\end{rem}

The example in Remark \ref{rem-1.1}(a) shows that the sesquianalytic 
correlation $\expect\Phi(z)\bar\Psi(z)$ may be maximally big in the sense 
of modulus \emph{everywhere in the disk} $\D$. However, the example in
Remark \ref{rem-1.1}(c) only says that the analytic correlation 
$\expect\Phi(z)\Psi(z)$ may be maximal in modulus along the radius $[0,1[$ 
emanating from the origin. This leaves open the possibility of bounding 
$L^2$ means along concentric circles.
The fact that $\expect\Phi(z)\Psi(z)$ represents a holomorphic function 
in $\D$ limits to some extent the possible growth of the function.
However, from the work of Abakumov and Doubtsov \cite{AbDoub}, we see that this
is not a very strong restriction, and effectively knowing that 
$\expect\Phi(z)\Psi(z)$ is holomorphic does not add much to the growth control
beyond the pointwise bound \eqref{eq-twocorr3}, which may be understood as
belonging to a Korenblum-type growth space. For some other aspects on the
growth behavior of functions in Korenblum-type spaces, see \cite{BoLyuMalTho}.  
To measure growth of functions in the Bloch space, the asymptotic variance 
of a function in the Bloch space has been studied 
(see \cite{McM1}, \cite{AIPP}, \cite{IvriiQC}, \cite{Hed1}). 
We recall that the \emph{Bloch space} $\calB(\D)$ consists of all 
complex-valued holomorphic functions $f:\D\to\C$ such that 
\[
\|f\|_{\calB}:=\sup_{z\in\D}(1-|z|^2)|f'(z)|<+\infty.
\]  
Naturally, this defines a seminorm on $\calB(\D)$, as constants get seminorm 
value $0$.
The \emph{asymptotic variance} of a function $f\in\calB(\D)$ is
the quantity
\begin{equation}
\sigma(f)^2:=\limsup_{r\to1^-}
\frac{1}{\log\frac{1}{1-r^2}}\int_\Te|f(r\zeta)|^2\diff s(\zeta).
\label{eq-asvar1}
\end{equation}
At least in dynamical situations, it captures very well the boundary growth 
of the given function. From a probabilistic point of view, it is based on
thinking of the evolution of the function $r\mapsto f(r\zeta)$ as a Brownian 
motion in time $\log\frac{1+r}{1-r}\sim\log\frac1{1-r^2}$. 
The analytic correlation $f(z)=\expect\Phi(z)\Psi(z)$ need not be an element
of the Bloch space $\calB(\D)$. However, it has a finite asymptotic variance
nevertheless.

\begin{thm}
For all jointly Gaussian processes $(\Phi,\Psi)$ consisting of
$\calD_0$-GAFs, we have the estimate
\[
\int_\Te |\expect\Phi(r\zeta)\Psi(r\zeta)|^2\diff s(\zeta)
\le 2r^2\log\frac{1}{1-r^2}+r^2.
\]
\label{thm-fund2}
\end{thm}

This means that in the $L^2$-average sense on concentric circles, the
function $\expect\Phi(z)\Psi(z)$ spends most of its time on $|z|=r$ with
values bounded by a constant times \emph{the square root of} 
$\log\frac{1}{1-r^2}$, which is of course much smaller than what the
bound \eqref{eq-twocorr3} would allow for. In terms of the random variables
$\alpha_j,\beta_k$, the left-hand side expresssion in the above theorem
equals
\begin{equation}
\int_\Te |\expect\Phi(r\zeta)\Psi(r\zeta)|^2\diff s(\zeta)
=\sum_{l=2}^{+\infty}r^{2l}\bigg|
\sum_{j,k:j+k=l}(jk)^{-\frac12}\langle\alpha_j,\bar\beta_k\rangle_\Omega\bigg|^2.
\label{eq-asvar1.01}
\end{equation}
It is natural to wonder if the bound $\sigma(f)^2\le2$ for the asymptotic 
variance of the analytic correlation $f(z)=\expect\Phi(z)\Psi(z)$ in 
Theorem \ref{thm-fund2} is optimal. By a construction due to Zachary 
Chase \cite{Chase}, we have the following. 

\begin{thm} {\rm(Chase)}
There is a permutation $\pi:\Z_+\to\Z_+$ such that if 
$\beta_j=\bar \alpha_{\pi(j)}$ and $f(z)=\expect\Phi(z)\Psi(z)$, we have
$\sigma(f)^2\ge 1.72$.
\label{thm-1.72}
\end{thm}

So, it remains to investigate the universal quantity 
$\Sigma^2_{\mathrm{oper}}:=\sup_f\sigma(f)^2$, where $f$ runs over all 
possible analytic correlations $\expect\Phi(z)\Psi(z)$. The subscript
refers to the relation with norm contractive operators on $L^2(\D)$
described in Subsection \ref{subsec-symbols1.1} below. 

\subsection{Orthonormal systems in separable Hilbert space}
In terms of the inner products $\langle\alpha_j,\bar \beta_k\rangle_\Omega$,
the condition that the elements belong to a Gaussian Hilbert space is 
inconsequential and may be removed.

\begin{cor}
Let $\{x_1,x_2,x_3,\ldots\}$ an $\{y_1,y_2,y_3,\ldots\}$ be orthonormal systems 
in a separable complex Hilbert space $\calH$. Then, for $0\le r<1$, we have
the estimate
\begin{equation*}
\sum_{l=2}^{+\infty}r^{2l}\bigg|\sum_{j,k:j+k=l}(jk)^{-\frac12}\langle 
x_j,y_k\rangle_{\calH}\bigg|^2\le 2\log\frac{\e}{1-r^2}.
\end{equation*}
\label{cor-Hilb1}
\end{cor}

One possible interpretation of the corollary is that \emph{on average}, 
the sums
\[
\bigg|\sum_{j,k:j+k=l}\bigg(\frac{l}{jk}\bigg)^{\frac12}\langle x_j,
y_k\rangle_\calH\bigg|^2
\]
are bounded by $2$. 


\subsection{The analytic correlation and Dirichlet operator 
symbols}
\label{subsec-symbols1.1}

For $w\in\D$, let $\szeg_w$ denotes the 
Szeg\H{o} kernel
\begin{equation}
\szeg_z(\zeta):=\frac{1}{1-\bar z \zeta}.
\label{eq-szego1}
\end{equation}
For functions in the Bergman space $A^2(\D)$, taking the inner product 
with $s_\zeta$ is the same as finding the average
\begin{equation}
\langle f,\szeg_z\rangle_\D=\int_0^1 f(zt)\diff t,\qquad f\in A^2(\D).
\end{equation}

\begin{defn}
Let $\Tope$ be a bounded $\C$-linear operator on $L^2(\D)$. 
The \emph{Dirichlet operator symbol} associated with $\Tope$ 
is the function
\begin{equation*}
\Zsymb[\Tope](z,w):=\langle \Tope(\bar \szeg_z),\szeg_w\rangle_\D,
\qquad z,w\in\D,
\end{equation*}
which is holomorphic in $\D^2$, with diagonal restriction 
\begin{equation*}
\oslash \Zsymb[\Tope](z)=\langle \Tope(\bar \szeg_z),\szeg_z\rangle_\D,
\qquad z\in\D.
\end{equation*}
\end{defn}

\begin{rem}
If $\Tope=\Mop_\mu$, the operator of multiplication by $\mu\in L^\infty(\D)$,
then 
\begin{equation}
\oslash \Zsymb[{\Mop_\mu}](z)=
\langle \Mop_\mu(\bar \szeg_z),\szeg_z\rangle_\D=\int_\D
\frac{\mu(\xi)\diff A(\xi)}{(1-z\bar\xi)^2},\qquad z\in\D,
\label{eq-Bergman1}
\end{equation}
which shows that $\oslash\Zsymb[\Tope]$ is a generalization of the
Bergman projection to the setting of general bounded operators. There is
a way to write $\Zsymb[\Tope]$ which makes the analogy with \eqref{eq-Bergman1}
clearer:
\[
\Zsymb[\Tope](z,w)=\langle\Tope,\szeg_z\otimes\szeg_w\rangle_{\mathrm{tr}}.
\]
Here, we use the bilinear tensor product 
$(f\otimes g)(h)=\langle h,\bar g\rangle f$, and the notation
$\langle A,\Bop\rangle_{\mathrm{tr}}=\mathrm{tr}(\Aop\bar \Bop)=
\mathrm{tr}(\bar \Bop\Aop)$ for the trace inner product.
\end{rem}

The next result characterizes the analytic correlations $\expect\Phi(z)\Psi(w)$
as the Dirichlet symbols associated with contractions on $L^2(\D)$. 

\begin{thm}
\noindent{\rm(a)} 
Given a pair of jointly Gaussian $\calD_0$-GAFs $(\Phi(z),\Psi(z))$ 
there exists a norm contraction $\Tope:L^2(\D)\to L^2(\D)$ such that
\begin{equation*}
\expect\Phi(z)\Psi(w)
=zw\langle\Tope \bar \szeg_z,\szeg_w\rangle_\D\quad
\qquad z,w\in\D,
\leqno{\mathrm{(i)}}
\end{equation*}

\noindent{\rm(b)}  Given a norm contraction $\Tope$ 
on $L^2(\D)$, there exists a pair of jointly Gaussian $\calD_0$-GAFs 
$(\Phi(z),\Psi(z))$ 
such that {\rm(i)} holds.
\label{thm-transfer1}
\end{thm}

In particular, we see that in the sense of the theorem, the analytic 
correlations $\expect\Phi(z)\Psi(w)$ may be identified with the 
Dirichlet operator symbols of contractions on $L^2\D)$:
\[
\expect\Phi(z)\Psi(w)=zw \Zsymb[\Tope](z,w).
\]

\subsection{Analytic correlations and the Bloch space} 
The \emph{Bloch space} $\calB(\D)$ consists of all complex-valued holomorphic
functions $f:\D\to\C$ such that 
\[
\|f\|_{\calB}:=\sup_{z\in\D}(1-|z|^2)|f'(z)|<+\infty.
\]  
This defines a seminorm on $\calB(\D)$, since constants get seminorm $0$.

\begin{defn}
The \emph{mock-Bloch space} $\calB^{\text{mock}}(\D)$ is the space of 
functions 
\[
\big\{\oslash\Zsymb[\Tope]:\,\,\Tope\,\,\, 
\text{is a bounded operator on}\,\,\,L^2(\D)\big\}.
\] 
\end{defn}

This mock-Bloch space is naturally endowed with a norm, which equals the
infimum of $\|\Tope\|$ over all operators $\Tope$ representing the same symbol
$\oslash\Zsymb[\Tope]$.
All functions in $\calB(\D)$ are in $\calB^{\text{mock}}(\D)$. 
This is well-known an easy to see using multiplication operators 
$\Mop_\mu$, as in \cite{Hed1} (compare with \eqref{eq-Bergman1}). On the 
other hand, is $\calB^{\text{mock}}(\D)$ contained in $\calB(\D)$? This is 
answered in the negative by the following.

\begin{thm}
There exists a function $f\in\calB^{\mathrm{mock}}(\D)$ which is not in
$\calB(\D)$.
\label{thm-notbloch}
\end{thm}

It is known that $\calB(\D)$ is maximal among M\"obius-invariant spaces
\cite{RubTim}, so $\calB^{\text{mock}}(\D)$ cannot be M\"obius-invariant in
the standard sense. 
For a M\"obius automorphism $\phi:\D\to\D$, let 
\begin{equation}
\Uop_\phi f(z):=\phi'(z)f\circ\phi(z),\quad
\bar\Uop_\phi f(z):=\bar\phi'(z)f\circ\phi(z),
\label{eq-unitaries1.2}
\end{equation}
be the associated unitary transformations of $L^2(\D)$.

\begin{thm}
For a M\"obius automorphism $\varphi:\D\to\D$, and a bounded operator
$\Tope$ on $L^2(\D)$, we write
$\Tope_\phi:=\Uop_\phi\Tope\bar\Uop_\phi^*$, which has the same norm as $\Tope$.
If we write $\Wsymb[\Tope](z,w):=zw\Zsymb[\Tope](z,w)$ and 
$\oslash\Wsymb[\Tope](z):=z^2 \Zsymb[\Tope](z,z)$, we then have the identity
\[
\oslash\Wsymb[{\Tope_\phi}](z)=\oslash \Wsymb[\Tope]\circ\phi(z)-
\Wsymb[\Tope](\phi(z),\phi(0))-\Wsymb[\Tope](\phi(0),\phi(z))+
\oslash \Wsymb[\Tope](\phi(0)).
\] 
\label{thm-mockbloch}
\end{thm}

Typically, in M\"obius-invariant spaces, the correction after a M\"obius
transform amounts to the subtraction of an appropriate constant.
Here, we instead subtract a function in the Dirichlet space.

\begin{rem}
The mock-Bloch space is intimately connected with the Hankel forms on the 
Dirichlet space studied by Arcozzi, Rochberg, Sawyer, and Wick (see Subsection
6.2 of \cite{ARSW}). In a sense that space of Hankel forms is predual to 
the mock-Bloch space. To be more precise, let 
$b(z)=\sum_{l=2}^{+\infty}\hat b(l)z^l$, and observe that
\[
\int_\D \oslash\Wsymb[\Tope](z)\bar b'(z)\diff A(z)=
\sum_{j,k=1}^{+\infty}\frac{\overline{\hat b(j+k)}}{\sqrt{jk}}\,\langle \Tope 
\bar f_j,f_k\rangle_\D,
\]
where $f_j(z)=jz^{j-1}$, $j=1,2,3,\ldots$, is the standard orthonormal basis 
in $A^2(\D)$. This means that $b'$ is in the predual space of 
the mock-Bloch space if and only if the infinite matrix 
$\{(jk)^{-1/2}\hat b(j+k)\}_{j,k=1}^{+\infty}$ is trace class. This supplies the 
connexion with Theorem 8 of \cite{ARSW}.
\end{rem}

\subsection{Symbols of Grunsky operators}

Let $\varphi:\D\to\C$ be a univalent function. In other words, $\varphi$ is
a conformal mapping onto a simply connected domain. The associated \emph{Grunsky
operator} $\boldG_\varphi$ is given by the expression
\begin{equation}
\boldG_\varphi f(z):=\int_\D\bigg(
\frac{\varphi'(z)\varphi'(w)}{(\varphi(z)-\varphi(w))^2}
-\frac{1}{(z-w)^2}\bigg)\,f(w)\diff A(w),\qquad z\in\D.
\label{eq-Grunskyop1}
\end{equation}
It is well-known that $\boldG_\varphi$ is a norm contraction on $L^2(\D)$, 
and it maps into the Bergman space $A^2(\D)$. This contractiveness is 
called the \emph{Grunsky inequalities}, and in this form it was studied 
in, e.g.,  \cite{BarHed}. For a given $\varphi$, we may consider instead 
the normalized mapping
\[
\tilde\varphi(z)=\frac{\varphi(z)-\varphi(0)}{\varphi'(0)},
\]
which has $\tilde\varphi(0)=0$ and $\tilde\varphi'(0)=1$. It is easy 
to see that $\boldG_{\tilde\varphi}=\boldG_\varphi$, so we might as well 
replace $\varphi$ by its normalized variant $\tilde\varphi$, and 
require of $\varphi$ that $\varphi(0)=0$ and $\varphi'(0)=1$. 
The Dirichlet symbol associated with $\boldG_\varphi$ is then
\begin{equation}
\Wsymb[\boldG_\varphi](z,w)=zw\,\Zsymb[\boldG_\varphi](z,w)=
\log\frac{zw(\varphi(z)-\varphi(w))}{(z-w)\varphi(z)\varphi(w)},
\qquad(z,w)\in\D^2,
\label{eq-WboldG}
\end{equation}
with diagonal restriction
\[
\oslash \Wsymb[\boldG_\varphi](z)=z^2\,\oslash \Zsymb[\boldG_\varphi](z)
=\log\frac{z^2\varphi'(z)}{(\varphi(z))^2},\qquad z\in\D.
\]
We want to characterize the Dirichlet symbols of the above form 
\eqref{eq-WboldG} among all Dirichlet symbols $\Wsymb[\Tope](z,w)$ of 
norm contractions $\Tope$ on $L^2(\D)$. 

\begin{thm}
A function $Q=Q(z,w)$ which is holomorphic on $\D^2$ 
is of the form $\Wsymb[\boldG_\varphi](z,w)$ for a normalized univalent function 
$\varphi:\D\to\C$ if and only if 

\noindent{\rm (a)} $Q(0,w)\equiv0$ and $Q(z,0)\equiv0$, and 

\noindent{\rm (b)}  $Q=Q(z,w)$ solves the nonlinear wave equation
\[
\partial_z\partial_w Q+(\partial_z Q)(\partial_w Q)
-\frac{z^2\partial_z Q-w^2\partial_w Q}{zw(z-w)}=0.
\]
\label{thm-NLW}
\end{thm}

\begin{rem}
This result ties in nicely with deformation theory. Let $\Lop$ denote the 
linear wave operator 
\[
\Lop Q(z):=\partial_z\partial_w Q
-\frac{z^2\partial_z Q-w^2\partial_w Q}{zw(z-w)}.
\]
Let $\lambda\in\D$, and suppose we look for an analytic family of solutions
$\lambda\mapsto Q_\lambda$ to the above nonlinear wave equation 
$\Lop Q+(\partial_z Q)(\partial_w Q)=0$. If $Q_0\equiv0$, we Taylor expand 
$Q_\lambda=\sum_{j=1}^\infty\lambda^j\hat Q_j$ and see that the the nonlinear wave 
equation becomes a sequence of linear PDEs for the coefficient functions 
$\hat Q_j$.
First, $\hat Q_1$ solves the homogeneous equation $\Lop \hat Q_1=0$, while 
for $j=2,3,4,\ldots$, $\hat Q_j$ solves an inhomogeneous equation
$\Lop\hat Q_j=F$, where $F$ is a nonlinear expression involving 
the lower order coefficient functions $\hat Q_k$ for $1\le k<j$. 
\end{rem}

\begin{rem}
It is a matter of substantial interest whether $\Sigma_{\mathrm{conf}}^2:=
\sup_f \sigma(f)^2>1$, where the supremum is taken over all $f$ of the form
$f=\oslash\Wsymb[\boldG_\varphi]$ for a normalized univalent function 
$\varphi:\D\to\C$. This question is related to the issue of whether 
$\boldG_\varphi$ is special among the contractions, which it of course is
in accordance with Theorem \ref{thm-NLW}. On the other hand, for general
contractions, we have Chase's construction of Theorem \ref{thm-1.72}
which gives a rather big asymptotic variance $\approx1.72$.  
\end{rem}

\subsection{Acknowledgements and a comment}
This paper is the result of a joint project with Serguei Shimorin, of which
a preliminary version was available earlier \cite{HedShim3}. The present
treatment of the subject matter has evolved rather substantially since 
the preliminary writeup. Tragically Serguei passed away in July 2016 as the 
result of a hiking accident. We thank several colleagues who helped 
organizing a conference in his honor at the Mittag-Leffler Institute in 
June, 2018. Among the organizers were Catherine B\'en\'eteau, 
Dmitry Khavinson, Mihai Putinar, and Alan Sola. We also want to thank Eero 
Saksman for a conversation on the fact that the mock-Bloch space is bigger 
than the Bloch space, Oleg Ivrii and Bassam Fayad for their interest in 
the asymptotic variance, and Zachary Chase for his contribution with the 
construction of a permutation matrix with somewhat extremal properties.

\section{The duality induced by the bilinear form of GAF}

\subsection{The GAF as a duality}

Let us for the moment write $\Phi_{\balpha}(z)$ for the $\calD_0$-Gaussian
analytic function given by \eqref{eq-Gauss1}, having in mind the notation
$\balpha:=(\alpha_1,\alpha_2,\alpha_3,\ldots)$ for the Gaussian 
vector of elements from $\Gspace$. The closure in $\Gspace$ of the linear 
span of the vectors $\alpha_j$, $j=1,2,3,\ldots$, will be denoted by $\Xspace$.
We shall also need the closure in $\Gspace$ of the linear span of the vectors 
$\bar\alpha_j$, $j=1,2,3,\ldots$, and we denote it by $\Xspace_*$. 
The independence and symmetry of of these random variables means that 
the vectors $\alpha_j$ form an orthonormal system in $\Gspace$, and that
$\Xspace$ is orthogonal to $\Xspace_*$. 
 
Continuing along the same line of thinking, we would write 
$\Phi_{\boldsymbol\beta}(z)$ for $\Psi(z)$, the second copy of the same Gaussian 
process. Now, if $\Mop$ is a bounded linear operator on 
$\Xspace$, then $\Mop\alpha_j\in \Xspace$ and hence has a convergent 
expansion in basis 
vectors: 
\[
\Mop \alpha_j=\sum_{k=1}^{+\infty}M_{j,k}\alpha_k,
\]
where the sequence $k\mapsto M_{j,k}$ is in $l^2$. If we write $\Mop\balpha=
(\Mop\alpha_1,\Mop\alpha_2,\Mop\alpha_3,\ldots)$, we may speak of a Gaussian
analytic function process
\begin{equation}
\Phi_{\Mop\balpha}(z)=\sum_{j=1}^{+\infty}\Mop\alpha_j\, e_j(z)=\sum_{j=1}^{+\infty}
\sum_{k=1}^{+\infty}M_{j,k}\alpha_k\,e_j(z)=\sum_{k=1}^{+\infty}\alpha_k
\sum_{j=1}^{+\infty}M_{j,k} \,e_j(z)=\sum_{k=1}^{+\infty}\alpha_k\,
\Mop^\dagger e_k(z),
\label{eq-Gauss-M}
\end{equation}
where $e_j(z)=j^{-\frac12}z^j$ as before. Moreover, the \emph{GAF transpose of}
$\Mop$, given by
\begin{equation}
\Mop^\dagger e_k(z):=\sum_{j=1}^{+\infty}M_{j,k} e_j(z)
\label{eq-Gauss-MM}
\end{equation}
defines a bounded linear mapping on $\calD_0(\D)$, as it just corresponds to 
the transpose of the matrix for $\Mop$, and shifting the 
basis from that of the Gaussian space $\Xspace$ to that of $\calD_0(\D)$.  
This way we have a natural transpose mapping $\Mop\to\Mop^\dagger$, and
it is perhaps also natural to let its inverse be denoted the same way, so that
$(\Mop^{\dagger})^{\dagger}=\Mop$.

Typically, \eqref{eq-Gauss-M} will define a Gaussian analytic function with 
a correlation kernel which is different from that of $\Phi_{\balpha}(z)$. 
Indeed, while
$\expect\Phi_{\Mop\balpha}(z)\Phi_{\Mop\balpha}(w)=0$ automatically since 
$\Xspace$ is orthogonal to $\Xspace_*$, we see that
\begin{equation}
\expect\Phi_{\Mop\balpha}(z)\bar\Phi_{\Mop\balpha}(w)=\sum_{j,k=1}^{+\infty}
\langle\Mop\alpha_j,\Mop\alpha_k\rangle_\Omega\, e_j(z)\bar e_k(w),
\label{eq-Gauss-M2}
\end{equation}
which need not coincide with the corresponding correlation for $\Phi_{\balpha}$.
However, in the special case when the restriction $\Mop|_\Xspace=\Uop$ is 
unitary on $\Xspace$, so that $\Uop^*\Uop=\Iop$ on $\Xspace$, 
\eqref{eq-Gauss-M2} gives us  
\begin{equation}
\expect\Phi_{\Uop\balpha}(z)\bar\Phi_{\Uop\balpha}(w)=\sum_{j,k=1}^{+\infty}
\langle\Uop^*\Uop\alpha_j,\alpha_k\rangle_\Omega\, e_j(z)\bar e_k(w)
=\sum_{j=1}^{+\infty}e_j(z)\bar e_j(w)=\log\frac{1}{1-z\bar w},
\label{eq-Gauss-M3}
\end{equation}
that is, the same correlation structure as for $\Phi_{\balpha}(z)$. 
In other words, $\Phi_{\Uop\balpha}$ is another copy of the $\calD_0$-GAF. 
When $\Uop:\Xspace\to\Xspace$ is unitary, its GAF transpose $\Uop^\dagger$ 
acts unitarily on $\calD_0(\D)$, and the functions $\Uop^\dagger e_j(z)$ form 
an orthonormal basis for $\calD_0(\D)$. Naturally, this goes the other way 
around as well, that is, if a unitary transformation $\Vop$ on 
$\calD_0(\D)$ is given, this defines another unitary transformation 
$\Vop^\dagger$ on $\Xspace$ via \eqref{eq-Gauss-M} 
with $\Vop$ in place of $\Mop^\dagger$.
An important instance is when the unitary transformation on $\calD_0(\D)$ is
generated by a M\"obius automorphism $\phi$ of the disk $\D$.
If $\phi:\D\to\D$ is a M\"obius automorphism, then the operator 
$\Vop_\phi$ given by 
\[
\Vop_\phi f(z):=f\circ\phi(z)-f\circ\phi(0)
\]
is unitary on $\calD_0(\D)$ and therefore corresponds to a unitary 
transformation 
$\Vop_\phi^\dagger$ acting on $\Xspace$ such that
\begin{equation}
\Phi_{\Vop_\phi^\dagger\balpha}(z)=\sum_{j=1}^{+\infty}\Vop_\phi^\dagger\alpha_j\,e_j(z)=
\sum_{j=1}^{+\infty}\alpha_j\,\Vop_\phi e_j(z)=\sum_{j=1}^{+\infty}\alpha_j\,j^{-\frac12}
(\phi(z)^j-\phi(0)^j).
\label{eq-Vopduality}
\end{equation}

\subsection{GAF and Hankel-type duality}
We describe a variation on the above-mentioned GAF duality theme.
Suppose that instead $\Mop$ is a bounded linear operator 
$\Xspace\to\Xspace_*$ (like a Hankel operator). 
In the same fashion as before, we write 
\[
\Mop \alpha_j=\sum_{k=1}^{+\infty}M_{j,k}\bar\alpha_k,
\]
and obtain that
\begin{equation}
\Phi_{\Mop\balpha}(z)=\sum_{j=1}^{+\infty}\Mop\alpha_j\, e_j(z)=\sum_{j=1}^{+\infty}
\sum_{k=1}^{+\infty}M_{j,k}\bar\alpha_k\,e_j(z)=\sum_{k=1}^{+\infty}\bar\alpha_k
\sum_{j=1}^{+\infty}M_{j,k} \,e_j(z)=\sum_{k=1}^{+\infty}\bar\alpha_k\,
\Mop^\ddagger e_k(z),
\label{eq-Gauss-M'}
\end{equation}
with $\Mop^\ddagger$, \emph{the GAF-Hankel transpose of} $\Mop$, given by the 
analogue of \eqref{eq-Gauss-MM}, 
\begin{equation}
\Mop^\ddagger e_k(z):=\sum_{j=1}^{+\infty}M_{j,k}
e_j(z).
\label{eq-Gauss-M''}
\end{equation}
As with the GAF transpose, we let it be its own inverse, so that 
$(\Mop^\ddagger)^\ddagger=\Mop$. If $\Mop:\Xspace\to\Xspace_*$ is isometric
and onto, then $\Mop^\ddagger$ acts unitarily on $\calD_0(\D)$. 
On the other hand, if $\Vop$ is unitary on $\calD_0(\D)$, the $\calD_0$-GAF
\begin{equation}
\sum_{k=1}^{+\infty}\bar\alpha_k\,\Vop e_k(z)=\sum_{k=1}^{+\infty}\bar\alpha_k
\sum_{j=1}^{+\infty}V_{k,j} \,e_j(z)=\sum_{j=1}^{+\infty}
\sum_{k=1}^{+\infty}V_{k,j}\bar\alpha_k\,e_j(z)=
\sum_{j=1}^{+\infty}\Vop^\ddagger\alpha_j\,e_j(z),
\label{eq-Gauss-M'''}
\end{equation}
where
\begin{equation}
\Vop^\ddagger\alpha_j=\sum_{k=1}^{+\infty}V_{k,j}\bar\alpha_k.
\label{eq-Gauss-M'''}
\end{equation}


\subsection{Representation of the correlations 
$\expect \Phi(z)\Psi(w)$ and $\expect \Phi(z)\bar\Psi(w)$}

In view of the definitions of $\Phi(z)$ and $\Psi(w)$, we have that
\begin{equation}
\Phi(z)\Psi(w)=\sum_{j,k=1}^{+\infty}\frac{\alpha_j\beta_k}{\sqrt{jk}}z^{j}w^{k},
\label{eq-CORR1}
\end{equation}
so that taking expectations, we obtain that
\begin{equation}
\expect\Phi(z)\Psi(w)
=\sum_{j,k=1}^{+\infty}(jk)^{-\frac12}(\expect\alpha_j\beta_k)\,z^{j}w^{k}
=\sum_{j,k=1}^{+\infty}(jk)^{-\frac12}\langle\alpha_j,\bar\beta_k
\rangle_\Omega\,z^{j}w^{k},\qquad z,w\in\D.
\label{eq-CORR2}
\end{equation}
Next, let $\Sope:\mathfrak{G}\to\mathfrak{G}$ be the bounded linear 
operator which maps $\Xspace_*\to\Yspace_*$ according to 
$\Sope \bar\alpha_j=\bar\beta_j$ for $j=1,2,3,\ldots$, while $\Sope\gamma=0$ 
holds for all 
$\gamma\in\mathfrak{G}\ominus\mathfrak{X}_*=\Xspace\oplus\Nspace$. 
Then $\Sope$ is a partial isometry: it vanishes on $\Xspace\oplus\Nspace$, 
and acts isometrically on $\Xspace_*$.
In terms of this operator, we may rewrite \eqref{eq-CORR2}:
\begin{equation}
\expect\Phi(z)\Psi(w)=
\sum_{j,k=1}^{+\infty}(jk)^{-\frac12}\langle\alpha_j,\bar\beta_k
\rangle_\Omega\,z^{j}w^{k}=\sum_{j,k=1}^{+\infty}(jk)^{-\frac12}
\langle\alpha_j,\Sope\bar\alpha_k\rangle\,z^{j}w^{k},\qquad z\in\D.
\label{eq-CORR3}
\end{equation}
While the representation \eqref{eq-CORR3} has some good properties, it is not
too convenient to give useful estimates. We split 
\[
\bar\beta_j=\Sop\bar\alpha_j=\Pop_{\Xspace}\Sop\bar\alpha_j+
\Pop_{\Xspace}^\perp\Sop\bar\alpha_j\quad\Longleftrightarrow\quad
\beta_j=\bar\Sop\alpha_j=\Pop_{\Xspace_*}\bar\Sop\alpha_j+
\Pop_{\Xspace_*}^\perp\bar\Sop\alpha_j,
\]
so that the process $\Psi(w)$ takes the form
\begin{equation*}
\Psi(w)=\sum_{j=1}^{+\infty}\beta_j\,e_j(w)= 
\sum_{j=1}^{+\infty}\Pop_{\Xspace_*}\bar\Sop\alpha_j\,e_j(w)+
\sum_{j=1}^{+\infty}\Pop_{\Xspace_*}^\perp\bar\Sop\alpha_j\,e_j(w)
=:\Psi_1(w)+\Psi_2(w),
\end{equation*}
with the obvious splitting of the process in two. 
Since 
\[
\expect \Phi(z)\Psi_2(w)=\langle \Phi(z),\bar\Psi_2(w)\rangle_\Omega=0
\] 
as a consequence of the properties of the projections, we see that
\[
\expect\Phi(z)\Psi(w)=\expect\Phi(z)\Psi_1(w),
\]
and from the GAF-Hankel duality of \eqref{eq-Gauss-M'}, 
\[
\Psi_1(w)=\sum_{j=1}^{+\infty}(\Pop_{\Xspace_*}\bar\Sop\alpha_j)\,e_j(w)=
\sum_{j=1}^{+\infty}\bar\alpha_j\,(\Pop_{\Xspace_*}\bar\Sop)^\ddagger e_j(w).
\]
It is now immediate that
\begin{equation}
\expect\Phi(z)\Psi(w)=\expect\Phi(z)\Psi_1(w)
=\sum_{j=1}^{+\infty}e_j(z)\,(\Pop_{\Xspace_*}\bar\Sop)^\ddagger e_j(w),\qquad z\in\D.
\label{eq-Gauss-Hankel1}
\end{equation}
Turning our attention to the other correlation $\expect\Phi(z)\bar\Psi(w)$, 
we split
\[
\bar\beta_j=\Sop\bar\alpha_j=\Pop_{\Xspace_*}\Sop\bar\alpha_j+
\Pop_{\Xspace_*}^\perp\Sop\bar\alpha_j\quad\Longleftrightarrow\quad
\beta_j=\bar\Sop\alpha_j=\Pop_{\Xspace}\bar\Sop\alpha_j+
\Pop_{\Xspace}^\perp\bar\Sop\alpha_j,
\]
so that the process $\Psi(w)$ takes the form
\begin{equation*}
\Psi(w)=\sum_{j=1}^{+\infty}\beta_j\,e_j(w)= 
\sum_{j=1}^{+\infty}\Pop_{\Xspace}\bar\Sop\alpha_j\,e_j(w)+
\sum_{j=1}^{+\infty}\Pop_{\Xspace}^\perp\bar\Sop\alpha_j\,e_j(w)
=:\Psi_3(w)+\Psi_4(w),
\end{equation*}
with the obvious splitting of the process in two. Since
\[
\expect \Phi(z)\bar\Psi_4(w)=\langle \Phi(z),\Psi_4(w)\rangle_\Omega=0
\]
as a consequence of the properties of the projections, we find that
\[
\expect\Phi(z)\bar\Psi(w)=\expect\Phi(z)\bar\Psi_3(w).
\]
In addition, by the duality of \eqref{eq-Gauss-MM}, 
\[
\Psi_3(w)=\sum_{j=1}^{+\infty}(\Pop_{\Xspace}\bar\Sop\alpha_j)\,e_j(w)=
\sum_{j=1}^{+\infty}\alpha_j\,(\Pop_{\Xspace}\bar\Sop)^\dagger e_j(w).
\]
which gives the equality
\begin{equation}
\expect\Phi(z)\bar\Psi(w)
=\sum_{j=1}^{+\infty}e_j(z)\,\overline{(\Pop_{\Xspace}\bar\Sop)^\dagger e_j(w)},
\qquad z,w\in\D.
\label{eq-Gauss-Hankel2}
\end{equation}
To simplify the notation, we write $\Qop=(\Pop_{\Xspace_*}\bar\Sop)^\ddagger$
and $\Rop=(\Pop_{\Xspace}\bar\Sop)^\dagger$ which are both contractions on
$\calD_0(\D)$. Then our main formulas become, for $z,w\in\D$:
\begin{equation}
\expect\Phi(z)\Psi(w)=\sum_{j=1}^{+\infty}e_j(z)\,\Qop e_j(w),\qquad 
\expect\Phi(z)\bar\Psi(w)=\sum_{j=1}^{+\infty}e_j(z)\,\overline{\Rop e_j(w)}.
\label{eq-QR1}
\end{equation}

\section{Proofs of the fundamental bounds}

\subsection{The joint pointwise bound of correlations}

\begin{proof}[Proof of Proposition \ref{prop-triang1}]
Essentially, we just need to use the property that the $8\times8$ matrix
\eqref{eq-4X4.1} is positive semidefinite. Since for complex constants 
$a,b,c,d$,
\begin{multline*}
0\le\big|a\Phi(z)+b\bar\Phi(z)-c\Psi(w)-d\bar\Psi(w)\big|^2=
(|a|^2+|b|^2)|\Phi(z)|^2+(|c|^2+|d|^2)|\Psi(w)|^2
+2\re(a\bar b(\Phi(z))^2)
\\
-2\re(a\bar c\Phi(z)\bar\Psi(w))-2\re(a\bar d\Phi(z)\Psi(w))-
2\re(\bar b c\Phi(z)\Psi(w))-2\re(\bar b d\Phi(z)\bar\Psi(w))
+2\re(c\bar d (\Psi(w))^2),
\end{multline*}
the inequality survives after taking the expectation:
\begin{multline*}
0\le\expect\big|a\Phi(z)+b\bar\Phi(z)-c\Psi(w)-d\bar\Psi(w)\big|^2=
(|a|^2+|b|^2)\log\frac{1}{1-|z|^2}+(|c|^2+|d|^2)\log\frac{1}{1-|w|^2}
\\
-2\re((a\bar c+\bar bd)\expect\Phi(z)\bar\Psi(w))
-2\re((a\bar d+\bar b c)\expect\Phi(z)\Psi(w)).
\end{multline*}
In other words, we have the inequality
\begin{multline*}
2\re((a\bar d+\bar b c)\expect\Phi(z)\Psi(w))
+2\re((a\bar c+\bar bd)\expect\Phi(z)\bar\Psi(w))
\le
(|a|^2+|b|^2)\log\frac{1}{1-|z|^2}+(|c|^2+|d|^2)\log\frac{1}{1-|w|^2}.
\end{multline*}
We now restrict the values of our parameters, and assume that 
$b=\bar a$ and $d=\bar c$.
The above inequality then gives that
\begin{equation*}
2\re(ac\expect\Phi(z)\Psi(w))
+2\re(a\bar c\expect\Phi(z)\bar\Psi(w))
\le
|a|^2\log\frac{1}{1-|z|^2}+|c|^2\log\frac{1}{1-|w|^2}.
\end{equation*}
We write $ac=|ac|\omega_1$ and $a\bar c=|ac|\omega_2$, where 
$|\omega_1|=|\omega_2|=1$. Then 
\begin{equation*}
2\re(\omega_1\expect\Phi(z)\Psi(w))
+2\re(\omega_2\expect\Phi(z)\bar\Psi(w))\le
\frac{|a|}{|c|}\log\frac{1}{1-|z|^2}+\frac{|c|}{|a|}\log\frac{1}{1-|w|^2}.
\end{equation*}
On the right-hand side, we are free to minimize over $|a|$ and $|c|$, while
on the left-hand side, we are free to maximize over the (freely choosable)
unit vectors $\omega_1$ and $\omega_2$. After optimization, we arrive at
the asserted estimate.
\end{proof}

\subsection{The proof of the fundamental integral estimate}
\label{subsec-thm-fund1}
\begin{proof}[Proof of Theorem \ref{thm-fund1}]
The first observation is that by $L^2(\D)$-orthogonality, 
\[
\int_\D\big|a w\expect\Phi(z)\Psi'(w)+b\bar w
\expect\Phi(z)\bar\Psi'(w)\big|^2
\frac{\diff A(w)}{|w|^2}
=|a|^2\int_\D\big|\expect\Phi(z)\Psi'(w)\big|^2\diff A(w)+
|b|^2\int_\D\big|\expect\Phi(z)\bar \Psi'(w)\big|^2\diff A(w).
\]
Next, we observe that by the representation \eqref{eq-QR1} and the
norm contractive property of $\Qop$,
\[
\int_\D\big|\expect\Phi(z)\Psi'(w)\big|^2\diff A(w)=
\Big\|\sum_{j=1}^{+\infty}e_j(z)\Qop e_j\Big\|_\nabla^2\le
\Big\|\sum_{j=1}^{+\infty}e_j(z) e_j\Big\|_\nabla^2=\sum_{j=1}^{+\infty}|e_j(z)|^2
=\log\frac{1}{1-|z|^2}, 
\]
and, that analogously, by the norm contractive property of $\Rop$,
\[
\int_\D\big|\expect\Phi(z)\bar\Psi'(w)\big|^2\diff A(w)=
\Big\|\sum_{j=1}^{+\infty}\bar e_j(z)\Rop e_j\Big\|_\nabla^2
\le\Big\|\sum_{j=1}^{+\infty}\bar e_j(z) e_j\Big\|_\nabla^2=
\sum_{j=1}^{+\infty}|e_j(z)|^2=\log\frac{1}{1-|z|^2}. 
\]
The proof is complete.
\end{proof}

\section{Dirichlet symbols of contractions on 
$L^2(\D)$ and analytic correlations of GAFs}

\subsection{The correspondence between Dirichlet symbols
and the analytic correlation}
We show the indicated relationship between the analytic correlation $\expect
\Phi(z)\Psi(w)$ and the Dirichlet symbols $\Zsymb[\Tope](z,w)$ for contractions
$\Tope$ on $L^2(\D)$. 

\begin{proof}[Proof of Theorem \ref{thm-transfer1}]
We begin with part (a), so we are given the orthonormal systems 
$\{\alpha_j\}_j$ and $\{\beta_j\}_j$ in the Gaussian Hilbert space $\Gspace$, 
and need to construct the norm contractive operator $\Tope$ on $L^2(\D)$ 
with the indicated property.
We let $\Sope:\Gspace\to\Gspace$ be the bounded linear operator with 
$\Sope\bar\alpha_j=\bar\beta_j$ for $j=1,2,3,\ldots$ while 
$\Sope\gamma=0$ for all $\gamma\in\Gspace\ominus\Xspace_*$.
Given that $\Sope$ is a contraction, the product $\Pop_{\Xspace}\Sope$ is 
a contraction as well, and we may decompose
\[
\Pop_\Xspace\bar\beta_k=\Pop_{\Xspace}\Sope\bar\alpha_k=
\sum_{j=1}^{+\infty}A_{k,j}\alpha_j, 
\]  
where $\sum_j|A_{k,j}|^2\le1$. 
For $j=1,2,3,\ldots$, we write $f_j(z)=e_j'(z)=j^{\frac12}z^{j-1}$, which 
constitutes an orthonormal basis in $A^2(\D)$, and
put
\[
\Tope^* f_k=\sum_{l=1}^{+\infty}A_{k,l}\bar f_l,\qquad k=1,2,3,\ldots.
\] 
By linearity and norm boundedness of the matrix $(A_{j,k})_{j,k}$, this 
defines $\Tope^*$ on $A^2(\D)$. Then 
\[
\langle \bar f_j,\Tope^*f_k\rangle_\D=\sum_{l=1}^{+\infty}A_{k,l}
\langle \bar f_j,\bar f_l\rangle_\D=A_{k,j}=\sum_{l=1}^{+\infty}A_{k,l}
\langle\alpha_j,\alpha_l\rangle_\Omega=\langle\alpha_j,
\Pop_{\Xspace}\Sop\bar\alpha_k\rangle_\Omega
=\langle\alpha_j,\Sop\bar\alpha_k\rangle_\Omega=
\langle\alpha_j,\bar\beta_k\rangle_\Omega,  
\]
and since
\begin{equation}
\bar z\szeg_z(\zeta)=\frac{\bar z}{1-\bar z\zeta}=\sum_{j=1}^{+\infty}\bar z^j
\zeta^{j-1}=\sum_{j=1}^{+\infty}\bar e_j(z)f_j(\zeta),
\label{eq-kerneldecomp1}
\end{equation}
it now follows that
\[
zw\,\langle \bar\szeg_z,\Tope^*\bar\szeg_w\rangle_\D=\sum_{j,k=1}^{+\infty}
e_j(z)e_k(w)\langle \bar f_j,\Tope^*f_k\rangle_\D=\sum_{j,k=1}^{+\infty}
\langle\alpha_j,\bar\beta_k\rangle_\Omega e_j(z)e_k(w)=\expect\Phi(z)\Psi(w),  
\]
so that condition (i) holds if $\Tope$ is the adjoint of $\Tope^*$.
But to properly define $\Tope$, we need to extend $\Tope^*$ to all of 
$L^2(\D)$. To this end, we simply declare 
that $\Tope^*f=0$ holds for $f\in L^2(\D)\ominus A^2(\D)$.
It remains to check that so constructed, $\Tope^*$ is a contraction on
$L^2(\D)$, for then the adjoint $\Tope$ is contractive as well. 
For a polynomial $f\in A^2(\D)$, we decompose it as a finite sum 
$f=\sum_kb_k f_k$ where $\|f\|^2_{L^2(\D)}=\sum_k|b_k|^2$, and since 
$\Tope^* f=\sum_{l,k}A_{k,l}b_k \bar f_l$, we find that
\[
\|\Tope^*f\|_{L^2(\D)}^2=\sum_{l}\bigg|\sum_{k}A_{k,l}b_k\bigg|^2
=\bigg\|\Pop_{\Xspace}\Sope\sum_k b_k\bar\alpha_k\bigg\|^2\le
\bigg\|\sum_k b_k\bar\alpha_k\bigg\|^2=\sum_k|b_k|^2=\|f\|^2_{L^2(\D)},
\] 
and it follows that $\Tope^*$ defines a contraction on $A^2(\D)$ and hence 
in a second step on all of $L^2(\D)$. This concludes the demonstration of 
part (a).

We proceed with the remaining task of obtaining part (b), which amounts
to constructing the Gaussian Hilbert space 
$\Gspace$ and the sequence $\beta_j$ and associated partial isometry 
$\Sop$ for a given contraction $\Tope$ on $L^2(\D)$. We recall that $\Xspace$
and $\Xspace_*$ are two orthogonal subspaces in $\Gspace$. However, the
sum $\Xspace\oplus\Xspace_*$ need not be all of $\Gspace$. We will assume 
that $\Nspace:=\Gspace\ominus(\Xspace\oplus\Xspace_*)$ is \emph{separable and 
infinite-dimensional} 
which just amounts to considering a sufficiently big (separable) 
Gaussian Hilbert space $\Gspace$. We split $\Nspace=\Mspace\oplus\Mspace_*$,
where $\Mspace$ is the closed linear span of certain elements 
$\nu_1,\nu_2,\nu_3,\ldots$ of $\Nspace$, which are all i i d standard complex 
Gaussian variables (see Subsection \ref{subsec-CGHS}). The space $\Mspace_*$
is then the closed linear span of the complex conjugates $\bar\nu_1,\bar\nu_2,
\bar\nu_3,\ldots$.
As for notation, we will need the orthogonal (Bergman) projection 
$\Pop_{A^2}:L^2(\D)\to A^2(\D)$, and its conjugate $\bar\Pop_{A^2}$ defined by
\[
\bar\Pop_{A^2}(f)=\overline{\Pop_{A^2}(\bar f)}.
\]
We begin with the observation that
\[
\langle\Tope\bar f_j, f_k\rangle_\D=\langle\bar f_j, \Tope^* f_k\rangle_\D
=\langle\bar f_j, \bar\Pop_{A^2}\Tope^* f_k\rangle_\D.
\qquad j,k=1,2,3,\ldots, 
\]
We need to find i i d standard Gaussian vectors 
$\beta_1,\beta_2,\beta_3,\ldots$  in the Gaussian Hilbert space $\Gspace$ 
such that
\[
\expect\alpha_j\beta_k=
\langle\alpha_j,\bar\beta_k\rangle_\Omega=\langle\Tope\bar f_j, f_k\rangle_\D=
\langle\bar f_j, \bar\Pop_{A^2}\Tope^* f_k\rangle_\D,
\qquad j,k=1,2,3,\ldots,
\]
since by summing over $j,k$ we arrive at
\begin{multline*}
\expect\Phi(z)\Psi(z)=\sum_{j,k=1}^{+\infty}e_j(z)e_k(w)\,
\expect\alpha_j\beta_k=
\sum_{j,k=1}^{+\infty}e_j(z)e_k(w)\,\langle\Tope\bar f_j, f_k\rangle_\D
\\
=\sum_{j,k=1}^{+\infty}e_j(z)e_k(w)\,
\langle\bar f_j, \bar\Pop_{A^2}\Tope^* f_k\rangle_\D=
\langle\bar \szeg_z, \bar\Pop_{A^2}\Tope^* \szeg_w\rangle_\D
=\langle\bar\Pop_{A^2}\bar \szeg_z, \Tope^* \szeg_w\rangle_\D=
\langle\bar \szeg_z, \Tope^* \szeg_w\rangle_\D=\langle\Tope\bar \szeg_z, 
\szeg_w\rangle_\D,
\end{multline*}
where we used \eqref{eq-kerneldecomp1}.

The element $\bar\Pop_{A^2}\Tope^* f_k$ is in the space of complex conjugates
of $A^2(\D)$, and as such it has an expansion
\[
\bar\Pop_{A^2}\Tope^* f_k=\sum_{l=1}^{+\infty}A_{k,l}\bar f_l, 
\]
where $\sum_j|A_{k,j}|^2\le1$. We need $\Sope$ to have the property that
in terms of the above expansion,
\[
\Pop_{\Xspace}\Sope\bar\alpha_k=\Aop\bar\alpha_k:=
\sum_{j=1}^{+\infty}A_{k,j}\alpha_j, 
\]
which defines $\Aop$ as an operator $\Xspace_*\to\Xspace$. As such, 
it is a contraction.
Indeed, if $\gamma\in\Xspace_*$ has expansion $\gamma=\sum_k b_k \bar\alpha_k$, 
we obtain that
\[
\|\Aop\gamma\|_{\Omega}^2=
\sum_{j}\bigg|\sum_{k} A_{k,j}b_k\bigg|^2
=\bigg\|\bar\Pop_{A^2}\Tope\sum_k b_k\bar e_k\bigg\|^2\le
\bigg\|\sum_k b_k\bar e_k\bigg\|^2=\sum_k|b_k|^2=\|\gamma\|^2,
\]
which verifies the norm contractivity of $\Aop$.
We proceed to define the operator $\Sope$ and hence the Gassian vectors 
$\bar\beta_j=\Sop\bar\alpha_j$. 
To do this, we appeal to a standard procedure in operator theory. 
Since $\Aop$ maps $\Xspace_*\to\Xspace$, it has an adjoint
$\Aop^\circledast$ which maps $\Xspace\to\Xspace_*$. 
We now form the \emph{defect operator} 
\[
\Dop:=(\Iop_{\Xspace_*}-\Aop^\circledast\Aop)^{1/2},
\]
which maps $\Xspace_*\to\Xspace_*$.
The square root is well-defined given that we are taking the square root
of a positive (semidefinite) operator. We use this defect operator to define
an associated operator $\tilde\Dop$ on $\Mspace$, by declaring that 
if $\Dop\bar\alpha_j=\sum_k D_{j,k}\bar\alpha_k$, then
\[
\tilde\Dop\nu_j=\sum_{k}D_{j,k}\nu_k,\qquad j=1,2,3,\ldots.
\]  
Then $\tilde\Dop$ becomes a contraction on $\Mspace$, and we may now define
the operator $\Sope$. 
For $\gamma\in\Gspace\ominus\Xspace_*$, we declare 
$\Sope\gamma=0$.
For $\gamma\in\Xspace_*$, we expand in basis vectors
$\gamma=\sum_k b_k\bar\alpha_k$, and define the Gaussian vectors
\begin{equation}
\bar\beta_k=\Sope\bar\alpha_k:=\Aop\bar\alpha_k+\tilde\Dop\nu_k
\in\Xspace\oplus\Mspace,
\qquad k=1,2,3,\ldots,
\label{eq-defbeta1}
\end{equation}
where $\Pop_\Xspace\Sope$ is as before. 
Since $\tilde\Dop\nu_k\in\Mspace\subset\Nspace$, we see that
\[
\Pop_{\Xspace}\Sope\bar\alpha_k=\Pop_{\Xspace}\Aop\bar\alpha_k
+\Pop_{\Xspace}\tilde\Dop\nu_k=\Aop\bar\alpha_k,
\]
since $\Aop\bar\alpha_k\in\Xspace$ and we know that $\Nspace$ is orthogonal 
to $\Xspace$, so things are as they should be. Moreover, 
$\Sope$ acts isometrically on 
$\Xspace_*$, as we see from
\[
\|\Sope\gamma\|_{L^2(\D)}^2=
\|\Aop\gamma\|_{L^2(\D)}^2+\|\Dop\gamma\|^2=
\|\gamma\|^2.
\]
It follows that the functions $\bar\beta_k:=\Sop\bar\alpha_k$ form an
orthonormal system in $\Gspace$. It remains to verify that they are i i d
standard complex Gaussians, which requires in addition to orthonormality
that $\expect\bar\beta_j\bar\beta_k=0$ holds for all $j$ and $k$. In view of
\eqref{eq-defbeta1},
\[
\expect\bar\beta_j\bar\beta_k=\langle\bar\beta_j,\beta_k\rangle_\Omega=0,
\]
given that $\bar\beta_j\in \Xspace\oplus \Mspace$ while $\beta_k\in
\Xspace_*\oplus\Mspace_*$ and the subspaces $\Xspace\oplus\Mspace$ and
$\Xspace_*\oplus\Mspace_*$ are orthogonal to one another in $\Gspace$.
This tells us how to construct the sequence $\beta_1,\beta_2,\beta_3,\ldots$
stanrting from the contraction $\Tope$ on $L^2(\D)$, and concludes the
proof of part (b).
\end{proof}

\subsection{Orthonormal systems 
in Hilbert space and operator symbols}
We recall the setting of Corollary \ref{cor-Hilb1}, where $x_1,x_2,x_3,\ldots$
and $y_1,y_2,y_3,\ldots$ are two orthonormal systems in complex Hilbert space 
$\calH$. Let $\calX$ denote the closed linear span of the vectors 
$x_1,x_2,x_3,\ldots$, and $\Pop_\calX$ the orthogonal projection 
$\calH\to\calX$.   

\begin{proof}[Proof of Corollary \ref{cor-Hilb1}]
We recall the notation $f_j(z)=e_j'(z)=j^{1/2}z^{j-1}$, and let 
$\Tope^*$ be a linear operator with the property that 
\begin{equation}
\Tope^* f_j=\sum_{k}\langle y_j,x_k\rangle_\calH \bar f_k.
\label{eq-Tope*1}
\end{equation}
Then we have for scalars $c_j$ (only finitely many nonzero) that
\[
\bigg\|\Tope^*\sum_j c_jf_j\bigg\|^2_\D=
\bigg\|\sum_{j,k} c_j\langle y_j,x_k\rangle_\calH \bar f_k
\bigg\|^2_\D
=\bigg\|\sum_{j,k} c_j\langle y_j,x_k\rangle_\calH x_k
\bigg\|^2_\calH=\bigg\|\Pop_{\calX}\sum_{j} c_j y_j\bigg\|_\calH^2\le
\bigg\|\sum_{j} c_j f_j\bigg\|_\calH^2
\] 
which shows that $\Tope^*$ defines a norm contraction 
$A^2(\D)\to\mathrm{conj}\,A^2(\D)$. In a second step, we extend $\Tope^*$ to all
of $A^2(\D)$ by declaring that $\Tope^* f=0$ for all 
$f\in L^2(\D)\ominus A^2(\D)$, and we see that this defines a contraction on
$L^2(\D)$.
The Dirichlet symbol of $\Tope$ is then, in view of \eqref{eq-kerneldecomp1},
\[
zw\,\Zsymb[\Tope](z,w)=zw\langle \Tope\bar\szeg_z,\szeg_w\rangle_\D=
zw\langle \bar\szeg_z,\Tope^*\szeg_w\rangle_\D=
\sum_{j,k=1}^{+\infty}e_j(z)e_k(w)\langle \bar f_j,\Tope^*f_k\rangle_\D
=\sum_{j,k=1}^{+\infty}e_j(z)e_k(w)\langle x_j,y_k\rangle_\D.
\]
Taking the diagonal restriction, we have that
\[
z^2\Zsymb[\Tope](z,z)=\sum_{l=2}^{+\infty}z^l\sum_{j,k:j+k=l}(jk)^{-\frac12}
\langle x_j,y_k\rangle_\D,
\]
and it follows that the claim is a direct consequence of 
Theorem \ref{thm-fund2}.  
\end{proof}

\section{Hilbert spaces and diagonal restriction on 
the bidisk}
\label{sec-hilbspaces}

\subsection{Weighted Bergman spaces on the disk and bidisk}
For real $\alpha>-1$, we write $A^2_\alpha(\D)$ for the Hilbert space of 
holomorphic functions $f:\D\to\C$ subject to the norm boundedness condition
\[
\|f\|_{A^2_\alpha(\D)}^2=(\alpha+1)\int_\D|f(z)|^2(1-|z|^2)^\alpha\diff A(z)<+\infty.
\]
Moreover, we write $A^2_{-1,0}(\D^2)$ for the 
Hilbert space of holomorphic functions $f:\D\to\C$ subject to the norm 
boundedness condition
\[
\|f\|_{A^2_{-1,0}(\D)}^2=\int_\D\int_\Te |f(z,w)|^2\diff s(z)
\diff A(w)<+\infty.
\]
For analytic functions $f$ on the bidisk, we let $\oslash$ denote the operation
of taking the diagonal restriction, $\oslash f(z):=f(z,z)$. We may for
instance write $\partial_{z}^j\oslash(\partial_w^k f)$ to denote the function
\[
\partial_{z}^j\Big(\partial_w^k f(z,w)\big|_{w:=z}\Big).
\]
In \cite{HedShim1}, the following diagonal norm expansion theorem was 
obtained.

\begin{thm}
For $f\in A^2_{-1,0}(\D^2)$, we have that
\begin{equation*}
\|f\|_{A^2_{-1,0}(\D)}^2=\sum_{n=0}^{+\infty}
\frac{(n+2)_n}{(n+1)!}\bigg\|\sum_{k=0}^{n}
\frac{(-1)^k(k+2)_{n-k}}{k!(n-k)!(n+k+2)_{n-k}}
\partial_{z}^{n-k}\oslash(\partial_w^k f)\bigg\|^2_{A^2_{2n+1}(\D)}.
\end{equation*}
\label{thm-hedshim-DMJ}
\end{thm}

\subsection{The implementation of the fundamental estimate
into the diagonal norm expansion}
Our starting point is the instance of $(a,b)=(1,0)$ in Theorem 
\ref{thm-fund1}:
\[
\int_\D\big|a(z)\expect\Phi(z)\Psi'(w)\big|^2
\diff A(w)\le |a(z)|^2\log\frac{1}{1-|z|^2},\qquad z\in\D.
\]
We dilate each variable using $r$, $0<r<1$, multiply by $|a(z)|^2$ for some
$a\in H^2(\D)$, and integrate over $\Te\times\D$:
\[
r^2\int_\Te\int_{\D(0,\frac1r)}\big|a(z)\expect\Phi(rz)\Psi'(rw)\big|^2
\diff A(w)\diff s(z)\le \|a\|_{H^2}^2\,\log\frac{1}{1-r^2}.
\]
We now throw away a part of the domain of integration (but, by monotonicity,
we may remove the $r^2$ factor at the same time):
\begin{equation}
\int_\Te\int_{\D}\big|a(z)\expect\Phi(rz)\Psi'(rw)\big|^2
\diff A(w)\diff s(z)\le \|a\|_{H^2}^2\,\log\frac{1}{1-r^2}.
\label{eq-aest1}
\end{equation}
We recognize the left-hand side expression as the norm-square in the space
$A^2_{-1,0}(\D^2)$ of the function $f(z,w)=a(z)\expect\Phi(rz)\Psi'(rw)$.
Clearly,
\[
\oslash(\partial_w^k f)(z)=r^k a(z)\expect\Phi(rz)\Psi^{(k+1)}(rz),
\]
so an application of Theorem \ref{thm-hedshim-DMJ} gives that
\begin{multline}
\sum_{n=0}^{+\infty}
\frac{2(n+2)_n}{n!}\int_\D\bigg|\sum_{k=0}^{n}
\frac{(-1)^k(k+2)_{n-k}\,r^k}{k!(n-k)!(n+k+2)_{n-k}}
\partial_{z}^{n-k}\big(a(z)\expect\Phi(rz)\Psi^{(k+1)}(rz)\big)\bigg|^2
(1-|z|^2)^{2n+1}\diff A(z)
\\
\le \|a\|_{H^2}^2\,\log\frac{1}{1-r^2}.
\label{eq-expanded1}
\end{multline}
We choose for simplicity $a(z)\equiv1$, and expand the higher order
derivative using the Leibniz rule
\[
\partial_{z}^{n-k}\big(\expect\Phi(rz)\Psi^{(k+1)}(rz)\big)=
r^{n-k}\sum_{l=0}^{n-k}\frac{(n-k)!}{l!(n-k-l)!}\expect\Phi^{(n-k-l)}(rz)
\Psi^{(k+l+1)}(rz).
\]
It follows that
\begin{multline}
\sum_{k=0}^{n}\sum_{l=0}^{n-k}\frac{(-1)^k(k+2)_{n-k}\,r^k}{k!(n-k)!(n+k+2)_{n-k}}
\partial_{z}^{n-k}\big(\expect\Phi(rz)\Psi^{(k+1)}(rz)\big)
\\
=r^n\sum_{k=0}^{n}\sum_{l=0}^{n-k}\frac{(-1)^k(k+2)_{n-k}}{k!l!(n-k-l)!(n+k+2)_{n-k}}
\expect\Phi^{(n-k-l)}(rz)\Psi^{(k+l+1)}(rz)
\\
=r^n\sum_{m=0}^{n}\frac{(-1)^m(n+1)[(n-m+1)_m]^2}{m!(m+1)!(n+2)_n}
\big(\expect\Phi^{(n-m)}(rz)\Psi^{(m+1)}(rz)\big)
\label{eq-combid1}
\end{multline}
since it happens to be true for integers $m$ with $0\le m\le n$ that
\begin{equation*}
\sum_{k,l\ge0: k+l=m}\frac{(-1)^k(k+2)_{n-k}}{k!l!(n-m)!(n+k+2)_{n-k}}
=\frac{(-1)^m(n+1)[(n-m+1)_m]^2}{m!(m+1)!(n+2)_n}.
\end{equation*}
As we implement \eqref{eq-combid1} into \eqref{eq-expanded1}, we arrive at
\begin{multline*}
\sum_{n=0}^{+\infty}
\frac{2(n+1)^3\,r^{2n}}{(2n+1)!}\int_\D\bigg|\sum_{m=0}^{n}
\frac{(-1)^m[(n-m+1)_m]^2}{m!(m+1)!}
\big(\expect\Phi^{(n-m)}(rz)\Psi^{(m+1)}(rz)\big)\bigg|^2
(1-|z|^2)^{2n+1}\diff A(z)
\\
\le \log\frac{1}{1-r^2}.
\end{multline*}
If we only keep the first term with $n=0$ on the left-hand side we are left 
with
\begin{equation}
2\int_\D\big|\expect\Phi(rz)\Psi'(rz)\big|^2
(1-|z|^2)\diff A(z)
\le \log\frac{1}{1-r^2}.
\label{eq-PhiPsi1}
\end{equation}
We are free to switch the roles of $\Phi$ and $\Psi$, so that we also have
\begin{equation}
2\int_\D\big|\expect\Phi'(rz)\Psi(rz)\big|^2
(1-|z|^2)\diff A(z)
\le \log\frac{1}{1-r^2}.
\label{eq-PhiPsi2}
\end{equation}
Since 
\[
\partial_z\expect\Phi(rz)\Psi(rz)=
r\expect \Phi'(rz)\Psi(rz)+r\expect\Phi(rz)\Psi'(rz),
\]
it follows from \eqref{eq-PhiPsi1} and \eqref{eq-PhiPsi2} that
\begin{multline}
\int_\D\big|\partial_z\expect\Phi(rz)\Psi(rz)\big|^2
(1-|z|^2)\diff A(z)
\\
\le2r^2\int_\D\big(\big|\expect\Phi(rz)\Psi'(rz)\big|^2+
\big|\expect\Phi'(rz)\Psi(rz)\big|^2\big)(1-|z|^2)\diff A(z)
\le 2r^2\log\frac{1}{1-r^2}.
\label{eq-PhiPsi3}
\end{multline}

\begin{proof}[Proof of Theorem \ref{thm-fund2}]
A variant of the Littlewood-Paley identity states that
for an analytic function $f$ in the Hardy space $H^2(\D)$,
\[
\int_\D|f'(z)|^2(1-|z|^2)\diff A(z)=\int_\Te |f(z)|^2\diff s(z)-
\int_\D|f(z)|^2\diff A(z),
\]
so that with $F(z)=\expect\Phi(rz)\Psi(rz)$, \eqref{eq-PhiPsi3} asserts that
\begin{equation}
\int_\Te |F(rz)|^2\diff s(z)-\int_\D|F(rz)|^2\diff A(z)\le 
2r^2\log\frac{1}{1-r^2}.
\label{eq-formula-F}
\end{equation}
In terms of the Taylor expansion of $F$,
\[
F(z)=\sum_{j=2}^{+\infty}\hat F(j)z^j,
\]
the estimate \eqref{eq-formula-F} amounts to 
\begin{equation}
\sum_{j=2}^{+\infty}\frac{j\,r^{2j}}{j+1}\,|\hat F(j)|^2\le 
2r^2\log\frac{1}{1-r^2}.
\label{eq-formula-F2}
\end{equation}
By integration, we see from \eqref{eq-formula-F2} that
\begin{multline}
\int_\D|F(rz)|^2\diff A(z)=\sum_{j=2}^{+\infty}\frac{r^{2j}}{j+1}\,|\hat F(j)|^2
\le 2\int_0^r\sum_{j=2}^{+\infty}\frac{j\,t^{2j-1}}{j+1}\,|\hat F(j)|^2\diff t
\\
\le 2\int_0^r t\log\frac{1}{1-t^2}\diff t=(1-r^2)\log(1-r^2)+r^2\le r^2.
\label{eq-formula-F2}
\end{multline}
It now follows from \eqref{eq-formula-F} combined with the estimate
\eqref{eq-formula-F2} that
\begin{equation}
\int_\Te |F(rz)|^2\diff s(z)\le 
2r^2\log\frac{1}{1-r^2}+r^2,
\end{equation}
as claimed.
\end{proof}

\section{M\"obius invariance and the mock-Bloch space}

\subsection{M\"obius invariance of the Dirichlet symbol}
For a M\"obius automorphism $\phi$ of the unit disk $\D$, let $\Uop_\phi$ 
and $\Vop_\phi$ be the unitary transformations on $L^2(\D)$ given by 
\eqref{eq-unitaries1.2}. 
If $\phi,\psi$ are two such M\"obius automorphisms, we see that
\[
\Uop_\psi\Uop_\phi f=\Uop_\psi(\phi'(f\circ\phi))=\psi'(\phi'\circ\psi)
(f\circ\phi\circ\psi)=(\phi\circ\psi)'(f\circ\phi\circ\psi)=
\Uop_{\phi\circ\psi}(f),
\] 
which puts us in the context of representation theory. In particular, we find
that $\Uop_\phi^*=\Uop_\phi^{-1}=\Uop_{\phi^{-1}}$. 

\begin{lem}
We have that 
\[
\bar w\,\Uop_\phi^* \szeg_w=\bar\phi(w)\,\szeg_{\phi(w)}
-\bar\phi(0)\,\szeg_{\phi(0)},\qquad w\in\D.
\]
\label{lem-Mob1.1}
\end{lem}

\begin{proof}
This is a direct computation. 
\end{proof}

\begin{proof}[Proof of Theorem \ref{thm-mockbloch}]
In view of the definition of the operator 
$\Tope_\phi=\Uop_\phi\Tope\bar\Uop_\phi^*$, we see that
\begin{equation*}
\oslash W[\Tope_\phi](z)=z^2\langle\Uop_\phi\Tope\bar\Uop_\phi^*\bar\szeg_z,
\szeg_z\rangle_\D=z^2\langle\Tope\bar\Uop_\phi^*\bar\szeg_z,\Uop_\phi^*
\szeg_z\rangle_\D,
\end{equation*}
and by Lemma \ref{lem-Mob1.1}, it follows that
\begin{multline*}
z^2\langle\Tope\bar\Uop_\phi^*\bar\szeg_z,\Uop_\phi^*\szeg_z\rangle_\D
=\phi(z)^2\langle \Tope\bar\szeg_{\phi(z)},\szeg_{\phi(z)}\rangle_\D-\phi(0)\phi(z)
\langle \Tope\bar\szeg_{\phi(z)},\szeg_{\phi(0)}\rangle_\D
\\
-\phi(0)\phi(z)\langle \Tope\szeg_{\phi(0)},\szeg_{\phi(z)}\rangle_\D+
\phi(0)^2\langle \Tope\szeg_{\phi(0)},\szeg_{\phi(0)}\rangle_\D,
\end{multline*}
which is the claimed invariance.
\end{proof}

\subsection{The mock-Bloch space is bigger than the 
Bloch space}
We show that the product of two Dirichlet space functions need not 
be in the Bloch space. 

\begin{proof}[Proof of Theorem \ref{thm-notbloch}]
Let $r_1,r_2,r_3,\ldots$ be a increasing sequence on $]0,1[$ tending rapidly
to $1$. We let $f$ and $g$ be the functions
\[
f(z):=\sum_{j=1}^{+\infty}j^{-1}(1-r_j^2)\frac{z}{1-r_j z},\quad
g(z):=\sum_{j=1}^{+\infty}\frac{j^{-1}}{\sqrt{\log\frac{1}{1-r_j^2}}}
\log\frac{1}{1-r_j z}.
\]
Then 
\[
\|f\|_\nabla^2=\int_\D|f'|^2\diff A=\int_\D\bigg|\sum_{j=1}^{+\infty}j^{-1}
\tfrac{1-r_j^2}{(1-r_j z)^2}\bigg|^2\diff A(z)=\sum_{j,k=1}^{+\infty}(jk)^{-1}
\frac{(1-r_j^2)(1-r_k^2)}{(1-r_j r_k)^2}<+\infty
\]
if the sequence $\{r_j\}_j$ is sparse enough. In a similar manner, 
\[
\|g\|_\nabla^2=\int_\D|g'|^2\diff A=\int_\D\bigg|\sum_{j=1}^{+\infty}
\frac{j^{-1}}{\sqrt{\log\frac{1}{1-r_j^2}}}
\frac{r_j}{1-r_j z}\bigg|^2\diff A(z)=\sum_{j,k=1}^{+\infty}(jk)^{-1}
\frac{\log\frac{1}{1-r_jr_k}}
{\sqrt{\log\frac{1}{1-r_j^2}}\sqrt{\log\frac{1}{1-r_k^2}}}<+\infty
\]
if the sequence is sparse enough. We could require for instance that
simultaneously the following conditions should hold:
\[
\log\frac{1}{1-r_jr_k}\le 2^{-|j-k|}
\sqrt{\log\frac{1}{1-r_j^2}}\sqrt{\log\frac{1}{1-r_k^2}}
\] 
and
\[
\frac{1}{(1-r_jr_k)^2}\le 2^{-|j-k|}\frac{1}{(1-r_j^2)(1-r_k^2)}.
\]
By construction, we have
\[
f'(z)g(z)=\sum_{j,k=1}^{+\infty}(jk)^{-1}\frac{1-r_j^2}{(1-r_j z)^2}
\frac{\log\frac{1}{1-r_k z}}{\sqrt{\log\frac{1}{1-r_k^2}}}, 
\]
so that
\[
(1-r_l^2)f'(r_l)g(r_l)=\sum_{j,k=1}^{+\infty}(jk)^{-1}\frac{1-r_j^2}{(1-r_j r_l)^2}
\frac{\log\frac{1}{1-r_k r_l}}{\sqrt{\log\frac{1}{1-r_k^2}}}
\ge l^{-2}\sqrt{\log\frac{1}{1-r_l^2}} 
\]
which with a sufficiently sparse sequence $\{r_j\}_j$ can be made to tend
to infinity. Since both $f$ and $g$ have nonnegative Taylor coefficients, 
\[
(fg)'(x)=f'(x)g(x)+f(x)g'(x)\ge f'(x)g(x),\qquad 0\le x<1,
\]
so it would follow that 
\[
\|fg\|_{\calB}=\sup_{z\in\D}(1-|z|^2)|(fg)'(z)|\ge \sup_l(1-r_l^2)f'(r_l)g(r_l)
=+\infty.
\]
On the other hand, there is a rank $1$ operator $\Tope$ such that 
$f(z)g(z)=\oslash\Zsymb[\Tope](z)$, so $fg$ definitely belongs to the 
mock-Bloch space $\calB^{\mathrm{mock}}(\D)$.
\end{proof}

\section{Characterization of Dirichlet symbols of Grunsky 
operators}

\subsection{Grunsky operators}
Let $\varphi:\D\to\C$ be a univalent function. In other words, $\varphi$ is
a conformal mapping onto a simply connected domain. The associated 
\emph{Grunsky operator} $\boldG_\varphi$ is given by \eqref{eq-WboldG}, 
and it is 
well-known that $\boldG_\varphi$ is a norm contraction on $L^2(\D)$, 
and that it maps into the Bergman space $A^2(\D)$. This contractiveness is 
usually referred to as the \emph{Grunsky inequalities}, and in this form 
it was studied in, e.g., \cite{BarHed}. Without loss of generality, 
we assume that $\varphi(0)=0$ and $\varphi'(0)=1$. We recall that
the Dirichlet symbol associated with $\boldG_\varphi$ is given by 
\eqref{eq-WboldG}.

\begin{proof}[Proof of Theorem \ref{thm-NLW}]
We first show that any symbol $Q(z,w)=\Wsymb[\boldG_\varphi](z,w)$ for a 
normalized univalent function $\varphi$ has the properties (a) and (b). 
Since $\Wsymb[\boldG_\varphi](z,w)=zw\Zsymb[\boldG_\varphi](z,w)$ it follows
that (a) holds. We note that if $\psi(z):=1/\varphi(1/z)$ and if $\xi:=1/z$,
$\eta:=1/w$, then 
\begin{multline*}
Q(z,w)=\Wsymb[\boldG_\varphi](z,w)=\log\frac{zw(\varphi(z)-\varphi(w))}
{(z-w)\varphi(z)\varphi(w)}
=\log\frac{\xi^{-1}\eta^{-1}(\varphi(\xi^{-1})-\varphi(\eta^{-1}))}
{(\xi^{-1}-\eta^{-1})\varphi(\xi^{-1})\varphi(\eta^{-1})}=
\log\frac{\psi(\xi)-\psi(\eta)}{\xi-\eta}.
\end{multline*}
In other words, 
\[
\psi(\xi)-\psi(\eta)=(\xi-\eta)\,\e^{Q(\xi^{-1},\eta^{-1})},
\]
so that
\begin{multline}
0=\partial_\xi\partial_\eta (\psi(\xi)-\psi(\eta))
=\partial_\xi\partial_\eta\big\{(\xi-\eta)\,\e^{Q(\xi^{-1},\eta^{-1})}\big\}
\\
=\bigg\{\xi^{-2}\partial_z Q(\xi^{-1},\eta^{-1})
-\eta^{-2}\partial_w Q(\xi^{-1},\eta^{-1})+(\xi-\eta)\xi^{-2}\eta^{-2}
\big(\partial_z\partial_w Q(\xi^{-1},\eta^{-1})
\\
+(\partial_zQ(\xi^{-1},\eta^{-1}))
(\partial_zQ(\xi^{-1},\eta^{-1}))\big)
\bigg\}\,\e^{Q(\xi^{-1},\eta^{-1})}.
\label{eq-WE1}
\end{multline}
Changing back to $(z,w)$-coordinates, we obtain that
\begin{equation*}
0=z^2\partial_z Q(z,w)
-w^2\partial_w Q(z,w)+(w-z)zw
\big(\partial_z\partial_w Q(z,w)
+(\partial_zQ(z,w))
(\partial_zQ(z,w))\big),
\end{equation*}
which is the same as
\begin{equation*}
\frac{w^2\partial_w Q(z,w)-z^2\partial_z Q(z,w)}{(w-z)zw}=
\partial_z\partial_w Q(z,w)+(\partial_zQ(z,w))(\partial_zQ(z,w)),
\end{equation*}
that is, property (b). 

We turn to the reverse implication, to show that a holomorphic function $Q$
in $\D^2$ with the properties (a) and (b) is necessarily of the form $\Wsymb
[\boldG_\varphi]$ for some normalized conformal mapping $\varphi$. In view of
the above calculation \eqref{eq-WE1}, condition (b) asserts that 
\[
\partial_\xi\partial_\eta\big\{(\xi-\eta)\,\e^{Q(\xi^{-1},\eta^{-1})}\big\}=
0
\]
which means that locally in $\D_{\rm e}^2$, 
\[
(\xi-\eta)\,\e^{Q(\xi^{-1},\eta^{-1})}=G_1(\xi)+G_2(\eta),
\]
where $G_1,G_2$ are holomorphic but with possible logarithmic branching at
infinity. Letting $\eta\to\xi$, we find that 
$G_1(\xi)+G_2(\xi)=0$, so that $G_2(\eta)=-G_1(\eta)$. So the above identity 
becomes
\begin{equation}
(\xi-\eta)\,\e^{Q(\xi^{-1},\eta^{-1})}=G_1(\xi)-G_1(\eta).
\label{eq-G1G2}
\end{equation}
We still need to know that $G_1$ is a globally well-defined function in 
$\D_{\mathrm{e}}$ (without logarithmic branching).
We differentiate both sides with respect to 
$\xi$:
\[
G_1'(\xi)=\partial_{\xi}\big((\xi-\eta)\,\e^{Q(\xi^{-1},\eta^{-1})}\big)=
\big\{1-\xi^{-2}(\xi-\eta)\partial_z Q(\xi^{-1},\eta^{-1})\big\}\,
\e^{Q(\xi^{-1},\eta^{-1})}=\e^{Q(\xi^{-1},\xi^{-1})},
\]
where in the last step we plugged in $\eta=\xi$, which is allowed since
the expression is independent of $\eta$. As $|\xi|\to+\infty$, we have
$Q(\xi^{-1},\xi^{-1})=\Ordo(|\xi|^{-2})$, so that $\e^{Q(\xi^{-1},\xi^{-1})}=1+
\Ordo(|\xi|^{-2})$, which rules out a $\xi^{-1}$ term, and hence there is no
logarithmic branching. In addition, we see that $G_1'(\infty)=1$. If we put,
for some constant $c$, $\psi:=G_1+c$, then by \eqref{eq-G1G2}, 
\begin{equation*}
\e^{Q(\xi^{-1},\eta^{-1})}=\frac{\psi(\xi)-\psi(\eta)}{\xi-\eta}.
\end{equation*}
Since the left-hand side is holomorphic and does not vanish in 
$\D_{\mathrm{e}}^2$, it follows that $\psi$ is univalent on $\D_{\mathrm{e}}$. 
But then there must exist a point in the complex plane $\C$ which is not in the
image $\psi(\D_{\mathrm{e}})$, and by adjusting $c$ we can make sure that 
$0\notin \psi(\D_{\mathrm{e}})$. Then winding things backwards we get $\varphi$
from $\psi$ in the above fashion, and $Q(z,w)$ is seen to be of the form
\eqref{eq-WboldG}, as claimed. 
\end{proof}

\section{Zachary Chase's construction of a permutation}

\subsection{Permutation of bases} 
We consider a permutation $\pi:\Z_+\to\Z_+$. We use the permutation to define
that $\beta_j:=\bar\alpha_{\pi(j)}$, which in turn defines the second 
Gaussian process $\Psi(z)$. In this case, the formula \eqref{eq-asvar1.01}
reduces to
\begin{equation}
\int_\Te |\expect\Phi(r\zeta)\Psi(r\zeta)|^2\diff s(\zeta)
=\sum_{l=2}^{+\infty}r^{2l}\bigg(\sum_{j,k:j+k=l}(jk)^{-\frac12}\delta_{j,\pi(k)}\bigg)^2,
\label{eq-asvar1.02}
\end{equation}
where $\delta_{j,k}$ denotes the Kronecker delta, which equals $1$ if $j=k$ 
and $0$ otherwise. Since the sum of Kronecker deltas is squared, it makes 
sense to try to concentrate the times they equal $1$ to certain values of $l$.

\begin{proof}[Proof of Theorem \ref{thm-1.72}]
Let $d\ge 3$ be an integer. We define the permutation $\pi=\pi_d$ in terms
of a disjoint partition into intervals $\Z_+=I_1\cup I_2\cup I_3\cup\ldots$, 
where $I_m$ is an interval on $\Z_+$ which moves toward the right as $m$ 
increases.
On each interval $I_m$ we let $\pi_d$ permute the interval in question. 
The first interval is $I_1:=\{1,\ldots,d-1\}$, and we put $\pi_d(j):=d-j$ for 
$j\in I_1$. 
The second interval is $I_2:=\{d,\ldots,d^2-d\}$, and  we put 
$\pi_d(j):=d^2-j$ for $j\in I_2$. The third interval is 
$I_3:=\{d^2-d+1,\ldots,d^3-d^2+d-1\}$ and on it we put $\pi_d(j):=d^3-j$. 
The fourth interval is $I_4:=\{d^3-d^2+d,\ldots,d^4-d^3+d^2-d\}$, and on it
we put $\pi_d(j):=d^4-j$. The general formula is $\pi_d(j):=d^m-j$ on $I_m$,
but the endpoints of interval $I_m$ depend on whether $m$ is even or odd.
If $m$ is odd, then $m=2n-1$ for some $n=1,2,3,\ldots$, and 
\[
I_m=I_{2n-1}:=\bigg\{\frac{d^{2n-1}+1}{d+1},\ldots,\frac{d^{2n}-1}{d+1}\bigg\} 
\]
while if $m$ is even, then $m=2n$ for some $n=1,2,3,\ldots$, and
\[
I_m=I_{2n}:=\bigg\{\frac{d^{2n}+d}{d+1},\ldots,\frac{d^{2n+1}-d}{d+1}\bigg\}.
\] 
The permutation $\pi_d$ is now well-defined, and we see that for $k\in I_m$,
$\delta_{j,\pi_d(k)}=\delta_{j,d^m-k}=0$ unless $j+k=d^m$. This means that only 
the parameter values $l$ that are powers of $d$ contribute to the sum 
\eqref{eq-asvar1.02}. When $l=d^m$, we find that
\[
\sum_{j,k:j+k=d^m}(jk)^{-\frac12}\delta_{j,\pi_d(k)}=\sum_{j\in I_m}
j^{-\frac12}(d^m-j)^{-\frac12}=\frac{1}{d^m}\sum_{j\in I_m}
\bigg(\frac{j}{d^m}\bigg)^{-\frac12}\bigg(1-\frac{j}{d^m}\bigg)^{-\frac12}=
\int_{\frac{1}{d+1}}^{1-\frac{1}{d+1}}t^{-\frac12}(1-t)^{-\frac12}\diff t+\Ordo(d^{-m+1}),
\]
by thinking of the sum as the Riemann sum of the integral with step length
$d^{-m}$. The integral is the incomplete Beta function, since by symmetry
\[
\int_{\frac{1}{d+1}}^{1-\frac{1}{d+1}}t^{-\frac12}(1-t)^{-\frac12}\diff t
=\pi-2\int_{0}^{\frac{1}{d+1}}t^{-\frac12}(1-t)^{-\frac12}\diff t=
\pi-4(d+1)^{-\frac12}\,
{}_2F_1\big(\tfrac12,\tfrac12;\tfrac32;\tfrac1{d+1}\big),
\]
where the last equality relates it to the standard hypergeometric function.
As it is well-known that
\[
\lim_{r\to1^-}\frac{1}{\log\frac{1}{1-r^2}}
\sum_{m=1}^{+\infty}r^{2d^m}=\frac{1}{\log d},
\]
it follows from the obtained asymptotics that
\[
\lim_{r\to1^-}
\frac{1}{\log\frac{1}{1-r^2}}
\sum_{m=1}^{+\infty}r^{2d^m}
\bigg(\sum_{j,k:j+k=d^m}(jk)^{-\frac12}\delta_{j,\pi_d(k)}\bigg)^2
=\frac{1}{\log d}\,
\Big\{\pi-4(d+1)^{-\frac12}{}_2F_1
\big(\tfrac12,\tfrac12;\tfrac32;\tfrac1{d+1}\big)\Big\}^2.
\]
Finally, choosing $d=29$ gives us the value $\approx1.7208$. This is the 
asymptotic variance of the correlation function $f(z)=\expect\Phi(z)\Psi(z)$ 
with coefficients $\beta_j=\bar\alpha_{\pi_d(j)}$.
\end{proof}

\end{document}